\documentclass[reqno,11pt]{amsart}
\usepackage[T1]{fontenc}
\usepackage[utf8]{inputenc}
\usepackage[english]{babel}
\usepackage[left=2.5cm,right=2.5cm,bottom=2.2cm,top=2.2cm]{geometry}
\usepackage{graphicx,amsmath,amsfonts,amssymb,amsthm,thmtools}% Required for inserting 
\usepackage{dsfont}
\usepackage{mathtools,xcolor,color}
\usepackage{mathrsfs} 
\usepackage{enumitem}
\usepackage{float}
\numberwithin{equation}{section}
\usepackage{lipsum}
\usepackage{comment} 
\usepackage[hypertexnames=false,colorlinks=true,linkcolor=blue,citecolor=green!45!black,urlcolor=purple]{hyperref} %it correct the hyperlink using cleveref
\usepackage[nameinlink,noabbrev]{cleveref}

\newcommand{\sign}{\operatorname{sign}}
\renewcommand{\P}{\mathbb{P}}

\title{Rapid stabilization for Stokes and Navier--Stokes systems with localized disturbance } 
\author[C. Calle]{Christian Calle}
\author[P. Guzmán]{Patricio Guzmán} 
\address{Departamento de Matemática, Universidad Técnica Federico Santa María, Valparaíso, Chile.} 
\email{ccalle@usm.cl} 
\email{patricio.guzmanm@usm.cl} 

\author[H. Parada]{Hugo Parada} \address{Université de Lorraine, CNRS, Inria, IECL, F-54000 Nancy, France.} \email{hugo.parada@inria.fr} 

\author[J. Prada]{Jefferson Prada$^*$} \address{Facultad de Matemáticas, Pontificia Universidad Católica de Chile, Santiago, Chile.} \email{jefferson.prada@uc.cl} 
\date{\today}
\thanks{$^*$ Corresponding author.}
\thanks{C. Calle received financial support by ANID Becas/Doctorado Nacional 21252433. P. Guzm\'an received financial support from Fondecyt 11240290. H. Parada is currently supported by the Agence Nationale de la Recherche through the QuBiCCS project (ANR-24-CE40-3008).  J. Prada received financial support by ANID/Subdirección del Capital Humano/ Fondecyt Postdoctorado 2024-3250240.}
\newtheorem{theo}{Theorem}[section]
\newtheorem{lemma}[theo]{Lemma}
\newtheorem{prop}[theo]{Proposition}
\newtheorem{rem}[theo]{Remark}
\newtheorem{cor}[theo]{Corollary}
\newtheorem{defi}[theo]{Definition}

\begin{document}

\begin{abstract}
In this study, we analyze the rapid stabilization of Stokes and the two- and three-dimensional Navier--Stokes systems in the presence of an unknown spatially localized disturbance. A  multivalued feedback control strategy is proposed which combines the Lyapunov frequency method with the multivalued sign operator. This methodology connects Lyapunov-based stability analysis with spectral inequalities, while the inclusion of the sign operator provides robustness against disturbance. The closed-loop system is governed by a differential inclusion, for which well-posedness is proved via the theory of maximal monotone operators. This approach guarantees global rapid exponential stabilization for the Stokes system and local rapid exponential stabilization for the two- and three-dimensional Navier--Stokes systems, with arbitrarily large disturbance bounds. In three dimensions, we obtain global strong solutions and rapid exponential stabilization for initial data sufficiently small in the $L^2$-norm with a prescribed, arbitrarily large bound in the $H^1$-norm.
\end{abstract}

\maketitle

%\tableofcontents

\section{Introduction}
\label{sec:intro}
Let $\Omega$ be an open nonempty bounded connected set in $\mathbb{R}^n$, $n\geq2,$ with smooth boundary. Let the controlled domain $\omega\subset \Omega$ be a nonempty open subset. We are interested in the rapid stabilization of the $n$-dimensional incompressible Stokes system and the two- and three-dimensional incompressible Navier--Stokes systems with internal control and an unknown disturbance:
\begin{equation}\label{P_S}
\begin{cases}
\partial_ty-\Delta y+\nabla p = \mathds 1_\omega (u+d), & \text{ in } \Omega\times(0,\infty),\\
\nabla  \cdot y =0, & \text{ in } \Omega\times(0,\infty),\\
y=0, &\text{ on }\partial \Omega\times(0,\infty), \\
y(0)=y_0, & \text{ in } \Omega.
\end{cases}
\end{equation}
and 
\begin{equation}\label{P_N-S}
\begin{cases}
\partial_t y - \Delta y  +(y \cdot \nabla )y+\nabla p = \mathds 1_\omega (u+d), & \text{ in } \Omega\times(0,\infty),\\
\nabla  \cdot y =0, & \text{ in } \Omega\times(0,\infty),\\
y=0, &\text{ on }\partial \Omega\times(0,\infty), \\
y(0)=y_0, & \text{ in } \Omega.
\end{cases}
\end{equation}
Here for $T>0$, $y:\Omega\times(0,T)\to \mathbb{R}^n$ represents the velocity of a fluid, $p:\Omega\times(0,T)\to \mathbb{R}$ denotes the pressure of this fluid, $u:\omega\times (0,T)\to \mathbb{R}^n$ is the control acting in $\omega$ and $d:\omega\times (0,T)\to \mathbb{R}^n$ is an unknown disturbance, and $y_0:\Omega\to \mathbb{R}^n$ is a given initial condition. Also, let us denote by $\nu(x)$ the outward normal to $\partial\Omega$ at the point $x\in \partial\Omega$ and use the following notation
$$\mathcal V :=\{ y\in C_0^\infty (\Omega)^n : \nabla \cdot y =0, \ \text{ in } \Omega\}, \quad H:=\overline{\mathcal V}^{L^2(\Omega)^n}, \qquad V:=\overline{\mathcal V}^{H_0^1(\Omega)^n}.$$
$$\mathbf{L^2}(\Omega) = \{ L^2(\Omega)\}^n, \quad \mathbf{H_0^1}(\Omega)=\{H_0^1(\Omega)\}^n, \quad \mathbf{H^2}(\Omega) =\{H^2 (\Omega) \}^n.$$
In particular, we have by \cite[Chapter~I, Theorems~1.4 and~1.6]{temam84}, these spaces admit the following characterizations:
\begin{equation*}
\begin{gathered}
H=\{y\in \mathbf{L^2}(\Omega) \ | \  \nabla\cdot y =0 \text{ in } \Omega, \ y\cdot \nu =0 \text{ on } \partial\Omega \},\\
V= \{ y\in\mathbf{H_0^1}(\Omega) \ | \  \nabla\cdot y = 0 \text{ in } \Omega\}.
\end{gathered}
\end{equation*}
Here, the divergence is understood in the sense of distributions, and the boundary condition in the characterization of $H$ is understood in the sense of the normal trace.\\

We seek to achieve exponential stabilization at a prescribed decay rate despite the unknown disturbance $d$. The feedback will depend on the state and on an a priori bound for the disturbance, without requiring knowledge of $d$ itself. Let $D>0$ be a fixed known constant. We assume that the unknown disturbance $d$ satisfies:
\begin{enumerate}[label=\text{(A\arabic*)}]
\item\label{A_1} $d\in L^1_{\mathrm{loc}}(0,\infty;\mathbf{L^2}(\omega))$.
\item\label{A_2} $\|d(\cdot,t)\|_{\mathbf{L^2}(\omega)}\leq D$ for almost every $t>0$.
\end{enumerate}
The problem considered here lies within the broader study of control and stabilization of distributed parameter systems. A foundational contribution to this theory was made by J.-L. Lions \cite{lions1971,lions1988}, who established a functional analytic framework for the controllability and stabilization of distributed parameter systems. From this perspective, several feedback stabilization methods were later developed, including spectral techniques, Riccati-based approaches, and backstepping transformations, among others.\medskip

In the parabolic setting, stabilization has been developed through a variety of spectral and feedback methods. Triggiani \cite{triggiani1975,triggiani1980} established abstract stabilizability and boundary feedback results based on compact resolvent and spectral arguments, later extended to nonlinear and semilinear systems by Barbu and Triggiani \cite{bt2004} and Barbu and Wang \cite{barbuwang2003}. Further approaches include Riccati-based feedback laws for reaction--diffusion systems \cite{breiten2014}, spectral characterizations via the Fattorini criterion \cite{badra2014}, and, for the heat equation, boundary feedback, Volterra and Fredholm backstepping transformations, finite-dimensional spectral controllers, and linear-quadratic feedback laws \cite{l2003,gp2020,ks2008}. Xiang~\cite{xiang2023small,X2024} developed the frequency Lyapunov method, which combines a Lebeau--Robbiano spectral inequality \cite{lr1995} with a finite-dimensional feedback acting on the low-frequency modes. This yields exponential stabilization with an arbitrarily prescribed decay rate, together with explicit estimates on the stabilization cost. The same strategy also provides small-time null controllability and finite-time stabilization, and has been developed in several settings, including heat flows \cite{coronxiang2025}, stochastic problems \cite{hlbp2026}, and boundary control \cite{xiao2025}.\medskip

% Disturbances and Methods
Robust stabilization becomes substantially more delicate when the system is affected by perturbations whose size, structure, or channel of action is not fully known. In this context, several strategies have been developed depending on how the disturbance interacts with the dynamics. For matched disturbances, where the perturbation enters through the control channel, \textit{Sliding Mode Control} provides robustness by driving trajectories onto a suitable sliding manifold in finite time \cite{bmp2025,gj2013,gl2014}. \textit{Active Disturbance Rejection Control} has also been used to stabilize multidimensional heat equations with disturbances entering through the control channel \cite{zj2017}. For disturbances generated by a finite-dimensional exosystem, an approach based on the internal model principle provides robust output regulation for boundary-controlled multidimensional heat equations \cite{zgp2025}. Complementing these general frameworks, Guzmán and Hernández \cite{gh2023,gh2025} stabilized unstable heat equations under unknown boundary disturbances by combining backstepping, Lyapunov methods, and multivalued sign feedbacks. More recently, multidimensional heat equations with unknown localized disturbances were considered in \cite{guzman2025rapid}, combining Xiang’s frequency Lyapunov method with a multivalued sign operator to obtain rapid stabilization through a well-posed closed-loop differential inclusion.\medskip

By way of illustration, without being exhaustive, Table \ref{table:disturbance} lists some of the works addressing disturbances consulted in the literature.
\begin{table}[H]
\centering
\begin{tabular}{c|c|c}
Equation & Distributed disturbance & Boundary disturbance  \\
\hline
Heat & \cite{guzman2025rapid,zj2017,gp2020}& \cite{hcp2020,gh2023,gh2025,zgp2025} \\
Wave & \cite{fx2015,opu2011} & \cite{gj2013,gk2012,m2022,mgk2019,zw2018} 
\\
Beam & \cite{a2023} & \cite{gk2014} \\
%KdV & -- &
\end{tabular}
\caption{\label{literature:disturba}Stabilization of partial differential equations subjected to unknown disturbances.}
\label{table:disturbance}
\end{table}

For the Navier--Stokes equations, feedback stabilization has a long and well-established history. In Table \ref{literature} we present a selection of  relevant works.
\begin{table}[H]
\centering
\begin{tabular}{c|c|c}
System & Distributed control & Boundary control \\
\hline
Navier--Stokes 2D
& \cite{b2003,b2008,xiang2023small,lasieckapriyasadtriggiani2021}
& \cite{blt2006,badratakahashi2011} \\[2pt]
Navier--Stokes 3D
& \cite{b2003,b2008,bt2004,brs2011internal,breitenkunisch2020,lasieckapriyasadtriggiani2021,azmi2022}
& \cite{blt2006,f2004,raymond2007,badratakahashi2011} \\[2pt]
Stochastic Navier--Stokes (2D--3D)
& \cite{b2011,bp2011} & \\
\end{tabular}
\caption{Selected works on feedback stabilization of Navier--Stokes equations, classified by the location of the control.}
\label{literature}
\label{tl}
\end{table}
Feedback stabilization must be accompanied by a well-posedness argument for the nonlinear closed-loop system, this presents some challenges in the three dimensions case. Barbu and Triggiani \cite{bt2004} obtained local stabilization by finite-dimensional internal feedback in the fractional Stokes space $D(S^{1/4})$. Barbu, Rodrigues and Shirikyan \cite{brs2011internal} treated stabilization to nonstationary trajectories for small initial deviations in $V$. For boundary control, we refer to Barbu, Lasiecka and Triggiani \cite{blt2006} and Raymond \cite{raymond2007}. Badra and Takahashi \cite{badratakahashi2011} developed an abstract stabilization framework for nonlinear parabolic systems with applications to Navier--Stokes equations.

The functional setting is particularly important for the passage from linear stabilization to nonlinear well-posedness. Breiten and Kunisch \cite{breitenkunisch2020} constructed stabilizing feedbacks using generalized Lyapunov equations and established strong closed-loop well-posedness for small data in $V$. Lasiecka, Priyasad and Triggiani \cite{lasieckapriyasadtriggiani2021} proved local well-posedness and uniform stabilization in critical $L^q$-based Sobolev and Besov spaces using finite-dimensional localized internal controls. Receding horizon stabilization of three-dimensional flows to reference trajectories was studied by Azmi \cite{azmi2022}.

Robust control and feedback regularization provide further relevant precedents. Barbu and Sritharan \cite{barbusritharan1998} developed an $H^\infty$ theory for fluid dynamics, including a three-dimensional extension, with disturbance attenuation formulated through an energy-gain estimate. Balogh, Gilliam and Shubov \cite{baloghgilliamshubov1997} studied feedback regularization of Navier--Stokes equations. Here, our objective is exact rejection of an unknown matched disturbance together with a prescribed exponential decay rate. The disturbance need not decay in time or be small in amplitude.\medskip

Motivated by \cite{guzman2025rapid,xiang2023small}, we investigate whether an explicit and robust rapid-stabilization mechanism can be extended to the Navier--Stokes system. This is nontrivial due to incompressibility, pressure, the vector-valued Stokes dynamics, nonlinear convection, and the presence of an unknown localized disturbance. In this direction, this work develops an explicit rapidly stabilizing feedback for the two- and three-dimensional Navier--Stokes equations with an unknown spatially localized disturbance. In three dimensions, we combine the decay in $H$ with a bound in $V$, retaining the full Stokes dissipation in the higher-order energy estimate. This allows any prescribed disturbance bound $D$, at the cost of asking for a smaller initial $H$ norm.\medskip

We now state our main results. The first establishes disturbance rejection and rapid exponential stabilization
for the Stokes system.
\begin{theo}
\label{thm:sta-Stokes}
Let $n\geq2$, $\lambda>0$. There exist a multivalued feedback law $G_\lambda:H\longrightarrow 2^{\mathbf{L^2}(\Omega)}$
and a constant $C_\lambda\geq1$ such that, for every
disturbance $d$ satisfying \textup{\ref{A_1}--\ref{A_2}} and every $y_0\in H$, the closed-loop
Stokes system
\begin{equation}\label{eq:closed_loop_Stokes}
\begin{cases}
\partial_t y-\Delta y+\nabla p=\mathds 1_\omega (G_\lambda(y)+d), & \text{in }\Omega\times(0,\infty),\\
\nabla\cdot y=0, & \text{in }\Omega\times(0,\infty),\\
y=0, & \text{on }\partial\Omega\times(0,\infty),\\
y(0)=y_0, & \text{in }\Omega,
\end{cases}
\end{equation}
admits a unique generalized solution $y \in C([0,\infty);H)\cap L^2_{\mathrm{loc}}(0,\infty;V)$ satisfying,
\begin{equation}\label{eq:finalStokes}
\|y(t)\|_H\leq C_\lambda e^{-\frac{\lambda}{4} t}\|y_0\|_H,
\qquad t\geq0.
\end{equation}
\end{theo}
The second result establishes disturbance rejection and local rapid exponential stabilization for the two-dimensional Navier--Stokes system.
\begin{theo}
\label{thm:sta-NS} 
Let $n=2$, $\lambda>0$. There exist a multivalued feedback law
$G_\lambda:H\longrightarrow 2^{\mathbf{L^2}(\Omega)}$ and constants $C_\lambda\geq1$ and $\rho_\lambda>0$ such that, for each disturbance $d$ satisfying \textup{\ref{A_1}--\ref{A_2}} and every $y_0\in H$ with $\|y_0\|_H\leq\rho_\lambda$, the closed-loop Navier--Stokes system
\begin{equation}\label{eq:closed_loop_Navier-Stokes}
\begin{cases}
\partial_t y-\Delta y+(y \cdot \nabla )y +\nabla p=\mathds 1_\omega (G_\lambda(y)+d), & \text{in }\Omega\times(0,\infty),\\
\nabla\cdot y=0, & \text{in }\Omega\times(0,\infty),\\
y=0, & \text{on }\partial\Omega\times(0,\infty),\\
y(0)=y_0, & \text{in }\Omega,
\end{cases}
\end{equation}
admits a unique generalized solution $y \in C([0,\infty);H)\cap L^2_{\mathrm{loc}}(0,\infty;V)$ satisfying,
\[ \|y(t)\|_H\le C_\lambda e^{-\frac{\lambda}{4} t}\|y_0\|_H, \quad t\geq0.\] 
\end{theo}

Our final result concerns the three-dimensional system.
\begin{theo}\label{thm:sta-NS-3D}
Let $n=3$, $\lambda>0$, $D>0$, and $R>0$. There exist a multivalued feedback law
$G_\lambda:H\longrightarrow 2^{\mathbf{L^2}(\Omega)}$ and constants $C_\lambda\geq1$ and $\rho_{\lambda,D,R}>0$ such that,
for every disturbance $d$ satisfying
\textup{\ref{A_1}--\ref{A_2}} and every $y_0\in V$ with
\begin{equation}\label{eq:smallness-NS-3D}
\|y_0\|_V\leq \frac R2,
\qquad
\|y_0\|_H\leq \rho_{\lambda,D,R},
\end{equation}
the closed-loop Navier--Stokes system \eqref{eq:closed_loop_Navier-Stokes} admits a unique global strong solution $y\in H^1(0,T;H)\cap L^2(0,T;D(S))\cap C([0,T];V)$, for every $T>0$, satisfying,
\begin{equation}\label{eq:decay-NS-3D}
\|y(t)\|_H\leq C_\lambda e^{-\frac{\lambda}{4}t}\|y_0\|_H, \qquad t\geq0,
\end{equation}
and the solution remains bounded in $V$, i.e., $\sup_{t\geq0}\|y(t)\|_V\leq \frac R{\sqrt2}$.
\end{theo}
The exact definitions of strong and generalized solutions are given in \Cref{def:strong-solution} and \Cref{def:generalized-solution}, respectively. The constant $C_\lambda$ in the previous results can be chosen explicitly as
\[C_\lambda=\sqrt{\mu_\lambda}=C_1e^{C_1\sqrt{\lambda}},\]
where $C_1$ is the constant in the spectral inequality \eqref{lebeau_chaves}. In particular, $C_\lambda$ depends only on $\Omega$, $\omega$, and $\lambda$.\medskip

This paper is organized as follows. In \Cref{sec:pre}, we introduce the functional setting and the preliminary tools used throughout the paper. In \Cref{sec:sta-Stokes}, we construct the feedback law and derive a priori stabilization estimates for the Stokes system. In \Cref{sec:WP-Stokes}, we establish well-posedness of the closed-loop Stokes system using maximal monotone operator theory and complete the proof of \Cref{thm:sta-Stokes}. In \Cref{sec:sta-NS-2D}, we derive the two-dimensional a priori stabilization estimate. In \Cref{sec:WP-NS-2D}, we first construct strong solutions to the truncated systems in dimensions two and three. We then remove the truncation in dimension two, establish well-posedness of the original system, and complete the stabilization proof of \Cref{thm:sta-NS}. Finally, in \Cref{sec:sta-NS-3D}, we combine the estimates in $H$ and $V$ to remove the truncation in dimension three and establish global strong well-posedness and rapid exponential stabilization, proving \Cref{thm:sta-NS-3D}.
%%%%%%%%%%%%%%%%%%%%%%%%%%%%%%%%%%%%%%%%%%%%%%%%%%%%%%%
%%%%%%%%%%%%%%%%%%%%%%%%%%%%%%%%%%%%%%%%%%%%%%%%%%%%%%
%%%%%%%%%%%%%%%%%%%%%%%%%%%%%%%%%%%%%%%%%%%%%%%%%%%%%%%
\section{Preliminaries}
\label{sec:pre}
\subsection*{Functional framework}
Here, we consider the usual incompressible fluid mechanics framework \cite{temam84}. According to Helmholtz decomposition, for any $u\in \mathbf{L^2}(\Omega)$ there exist unique $v\in H$ and $p\in H^{1}(\Omega)/\mathbb{R}$ such that $u=v+\nabla p$, which defines the orthogonal Leray projection $P$ on
$\mathbf{L^2}(\Omega)$:
\[\begin{array}{rcl}
P: \mathbf{L^2}(\Omega) &\longrightarrow&  H,\\
u &\longmapsto& P u:=u-\nabla p.
\end{array}\]
Notice that, for every $f\in\mathbf{L^2}(\Omega)$,
\[\|P(\mathds1_\omega f)\|_H \leq\|\mathds1_\omega f\|_{\mathbf{L^2}(\Omega)} =\|f\|_{\mathbf{L^2}(\omega)}.\]
We equip $V$ with the inner product and associated norm
\[(u,v)_V:=\int_\Omega\nabla u:\nabla v\,dx,
\qquad \|u\|_V^2:=\int_\Omega|\nabla u|^2\,dx.\]Let $S$ be the Stokes operator defined as
\[\begin{array}{rcl}
S:D(S) &\longrightarrow &H,\\
u& \longmapsto &Su :=-P\Delta u.
\end{array} \]
where $D(S)=\mathbf{H^2}(\Omega)\cap V$. %\hugo{como nuestras demostraciones de WP son a tarvez de la teoria de operadores ese el buen dominio} 
We also denote by $S:V\to V'$ its variational extension, defined by
\[\langle Su,v\rangle_{V',V}:=(u,v)_V,\qquad u,v\in V.\]
For $u\in D(S)$, this definition agrees with the Stokes operator
under the identification $H\subset V'$. The Stokes operator on $H$ is positive and self-adjoint, with compact resolvent. Hence, there exists an orthonormal basis $\{e_j\}_{j\geq1}$ of $H$ consisting of eigenvectors of $S$:
\[Se_j=\tau_j e_j,\quad j\ge1, \quad \text{ where } \quad  0<\tau_1\le \tau_2\le\cdots\le \tau_j\le\cdots, \quad \tau_j\to\infty.\]
Equivalently, for every $i\ge1$, there exists a pressure $\pi_i\in H^1(\Omega)/\mathbb R$ such that
\begin{equation*}
\left\{
\begin{array}{rcll}
-\Delta e_i+\nabla \pi_i &=& \tau_i e_i,& \text{in }\Omega,\\
\nabla \cdot e_i &=& 0,
& \text{in }\Omega,\\
e_i &=& 0,
& \text{on }\partial\Omega.
\end{array}\right. \qquad (e_i,e_j)_H=\delta_{ij}. \end{equation*}
For $\lambda>0$, we set $N=N(\lambda):=\#\{j\geq1:\tau_j\leq\lambda\}$. Let $H_N:=\operatorname{Vect}\{e_j\}_{j=1}^N\subset V$ be the low frequency subspace of $H$, and let $\mathbb P_N$ be the orthogonal projection onto $H_N$ and $\mathbb P_N^\perp u:=u-\mathbb P_Nu$. For $u\in H$, we write $u_j:=(u,e_j)_H$, so that
\[\mathbb P_Nu=\sum_{j=1}^N u_j e_j, \quad \|u\|_H^2=\sum_{j\geq1}|u_j|^2,\qquad
\|\mathbb P_Nu\|_H^2=\sum_{j=1}^N|u_j|^2,\qquad
\|\mathbb P_N^\perp u\|_H^2=\sum_{j>N}|u_j|^2.\]
Moreover, for every $u\in V$,
\[\|u\|_V^2=\sum_{j\geq1}\tau_j|u_j|^2,\qquad
\|\mathbb P_Nu\|_V^2
=\sum_{j=1}^N\tau_j|u_j|^2
\leq\|u\|_V^2.\]
For some $\mu_\lambda>1$, we define the operator 
\[\begin{array}{rcl}
\mathscr{C}_\lambda: H & \longrightarrow &H,\\
u& \longmapsto & \mathscr C_\lambda u:= \mu_\lambda \P_N u+\P_N^\perp u.
\end{array}\]
Then, $\mathscr C_\lambda$ is bounded, self-adjoint, and positive on $H$. We consider the weighted inner product $ (\cdot, \cdot)_{\mu}$ and the norm $\|\cdot\|_\mu$ in $H$ by
\begin{equation}
\label{eq:mu-C}
(u, v)_{\mu} := (u,\mathscr{C}_\lambda v)_H = \mu_{\lambda} (\P_N u, \P_N v)_H + (\P_N^\perp u, \P_N^\perp v)_H, \qquad \|u\|^2_\mu = (u, u)_\mu.
\end{equation}
This norm is equivalent to the natural norm of $H$, that is,
\[\|u\|_H^2 \leq\|u\|_\mu^2 \leq\mu_\lambda\|u\|_H^2.\]
Moreover, $\mathscr C_\lambda$ maps $V$ into itself and satisfies
\[\|\mathscr C_\lambda u\|_V\leq\mu_\lambda\|u\|_V,\qquad u\in V.\]
We define the multi-valued sign operator on $\mathbf{L^2}(\Omega)$ by
\[\begin{array}{rcl}
\sign_{\mathbf{L^2}(\Omega)} : \mathbf{L^2}(\Omega)& \longrightarrow & 2^{\mathbf{L^2}(\Omega)},\\
u & \longmapsto &\sign_{\mathbf{L^2(\Omega)}}(u)
=
\begin{cases}
\left\{\dfrac{u}{\|u\|_{\mathbf{L^2}(\Omega)}}\right\}, & u\neq 0,\\
\{\eta\in \mathbf{L^2}(\Omega):\|\eta\|_{\mathbf{L^2}(\Omega)}\leq 1\}, & u=0.
\end{cases}
\end{array}\]
We shall use the following classical properties.
\begin{itemize}
\item If $\eta\in\sign_{\mathbf L^2(\Omega)}(f)$,  then $\|\eta\|_{\mathbf L^2(\Omega)}\leq1$.
\item If $f\in \mathbf{L^2}(\Omega)$ and
$\eta\in\sign_{\mathbf{L^2}(\Omega)}(f)$, then $(\eta,f)_{\mathbf{L^2}(\Omega)}=\|f\|_{\mathbf{L^2}(\Omega)}$.
\item $\sign_{\mathbf{L^2}(\Omega)}$ is monotone, namely,
\[\eta_i\in\sign_{\mathbf{L^2}(\Omega)}(f_i),\ i=1,2
\quad\Longrightarrow\quad
(\eta_1-\eta_2,f_1-f_2)_{\mathbf{L^2}(\Omega)}\geq0.\]
\item The graph is strongly--weakly closed: if
$f_k\to f \text{ in } \mathbf{L^2}(\Omega),$ $\eta_k\rightharpoonup \eta \text{ in } \mathbf{L^2}(\Omega),$ $\eta_k\in \sign_{\mathbf{L^2}(\Omega)}(f_k),$ then $\eta\in\sign_{\mathbf{L^2}(\Omega)}(f).$
\end{itemize}

\subsection*{Spectral inequality}
The following spectral inequality was proven by Chaves-Silva and Lebeau in \cite{chaves2016spectral} and is used by Xiang in \cite{xiang2023small}.
\begin{theo}[{\cite[Thm.~3.1]{chaves2016spectral}}]
\label{th:spectral-ineq}
There exists a constant $C_1 \geq 1$, depending only on $(\Omega, \omega)$ that is independent of $\lambda>0$ such that, for any $\lambda>0$ and for all $y \in H$ the following inequality holds,
\begin{equation}
\begin{aligned}\label{lebeau_chaves}
\|\mathbb{P}_Ny\|_{\mathbf{L^2}(\omega)}^2 \geq C_1^{-1} e^{-C_1 \sqrt{\lambda}}\|\mathbb{P}_Ny\|^2_H.
\end{aligned}
\end{equation}
\end{theo}

\subsection*{Trilinear form}
Throughout this subsection, we assume that $n=2$ or $n=3$, with the dimension specified in each estimate. Let $b\in\mathcal L^3(\mathbf{H_0^1}(\Omega);\mathbb R)$ be defined by
\begin{equation}\label{eq:def-b}
b(u,v,w):=\sum_{i,j=1}^{n}\int_{\Omega}u_i\,\partial_i v_j\,w_j\,dx=\int_\Omega (u\cdot\nabla)v\cdot w\,dx.
\end{equation}
This trilinear form is skew-symmetric in its last two variables, that is, 
\[b(u,v,w)=-b(u,w,v), \quad  u\in V, \quad  v,w\in \mathbf{H_0^1}(\Omega), \]
and hence 
\begin{equation}\label{eq:cancel-b} 
b(u,v,v)=0, \quad  \quad  u\in V,\quad  v\in \mathbf{H_0^1}(\Omega).
\end{equation}
For $u,v\in V$, we define $B(u,v)\in V'$ by
\[\langle B(u,v),w\rangle_{V',V}=b(u,v,w), \quad w\in V,\]
and write $B(u):=B(u,u)$. We shall use the following classical estimates; for $n=2$ see \cite[Chapter~III, Lemma~3.4 and Lemma~3.8]{temam84}, and for $n=3$ see \cite[Chapter~III, Lemma~3.8, inequality~(3.148)]{temam84}
\begin{lemma}\label{lem:trilinear-estimates}
Let $\Omega$ be a bounded Lipschitz open set in $\mathbb R^2$. Then, for every $u,v,w\in V$,
\begin{align}
|b(u,v,w)| \le & 2^{1/2} \|u\|_H^{1/2}\|u\|_V^{1/2}\|v\|_V\|w\|_H^{1/2} \|w\|_V^{1/2},\label{ladyzhenskaya2d}\\
|b(u,u,v)| \leq & 2^{1/2}\|u\|_H \|u\|_V\|v\|_V,\label{3.56}\\
\|B(u)\|_{V'}\leq & 2^{1/2} \|u\|_H \|u\|_V.\label{estimate:B}
\end{align}
\end{lemma}
\begin{lemma}
Let $\Omega\subset\mathbb R^2$ be a bounded open set of class $C^2$. There exists a constant $C(\Omega)>0$ such that, for every
$u\in V\cap\mathbf{H^2}(\Omega)$, we have $B(u)\in H$ and
\begin{equation}\label{eq:B-2d}
\|B(u)\|_H\leq C(\Omega)\|u\|_H^{1/2}\|u\|_V\|Su\|_H^{1/2}.
\end{equation}
\end{lemma}

\begin{lemma}\label{lem:B-3d}
Let $\Omega\subset\mathbb R^3$ be a bounded open set of class $C^2$. There exists $C(\Omega)>0$ such that, for every $u\in D(S)$, we have $B(u)\in H$ and
\begin{equation}\label{eq:B-3d}
\|B(u)\|_H\leq C(\Omega)\|u\|_V^{3/2}\|Su\|_H^{1/2}.
\end{equation}
\end{lemma}
We also need the following basic estimate for the convective term in the closed-loop system.

\begin{lemma}\label{lem:finite-projection-convection}
Let $n=2$ or $n=3$. For every $\lambda>0$, there exists $M_\lambda\geq0$ such that
\[
\|\nabla\mathbb P_Nv\|_{L^\infty(\Omega)}\leq M_\lambda\|v\|_H,
\qquad v\in H.
\]
Consequently, setting $a_\lambda:=(\mu_\lambda-1)M_\lambda$, we have
\begin{equation}\label{eq:finite-projection-convection}
|\langle B(y),\mathscr C_\lambda y\rangle_{V',V}|
\leq a_\lambda\|y\|_H^3,\qquad y\in V.
\end{equation}
\end{lemma}
\begin{proof}
The Stokes eigenfunctions are smooth up to the boundary under our assumptions on $\Omega$. Thus we may take
\[
M_\lambda=\left(\sum_{j=1}^{N}\|\nabla e_j\|_{L^\infty(\Omega)}^2\right)^{1/2},
\]
with the value zero if $N=0$. The first assertion follows by Cauchy--Schwarz. By cancellation and skew-symmetry,
\[
|\langle B(y),\mathscr C_\lambda y\rangle_{V',V}|
=(\mu_\lambda-1)|b(y,\mathbb P_Ny,y)|
\leq(\mu_\lambda-1)\|\nabla\mathbb P_Ny\|_{L^\infty(\Omega)}\|y\|_H^2.
\]
\end{proof}
\subsection*{Notion of solutions}
In the literature on the Navier--Stokes equations, several notions of solution are used, depending on the analytical framework and the regularity of the initial data and the forcing term. In particular, the term \emph{strong solution} may refer to different regularity classes. To distinguish the notions used in this paper, we first recall the definition of a strong solution for an abstract differential inclusion; see \cite[Chapter~IV]{s1997} and \cite[Definition~4.1]{barbu2010}. Throughout this paper, we refer
to this notion as an \emph{abstract strong solution}.

\begin{defi}\label{def:abstract-strong-solution}
Let $X$ be a Hilbert space,
$\mathcal A:D(\mathcal A)\subset X\to 2^X$,
$T>0$, $y_0\in X$, and $f\in L^1(0,T;X)$.
A function $y\in C([0,T];X)$ is called an
\emph{abstract strong solution} of
$y'+\mathcal Ay\ni f$ with initial datum $y_0$ if
$y(0)=y_0$, $y\in W^{1,1}(\varepsilon,T;X)$ for every $\varepsilon\in(0,T)$, and
\[ y(t)\in D(\mathcal A), \qquad y'(t)+\mathcal Ay(t)\ni f(t)
\quad\text{in }X,\quad\text{for a.e. }t\in(0,T)\]
\end{defi}

For the same abstract inclusion, we also introduce the notion
of an abstract generalized solution.

\begin{defi}\label{def:abstract-generalized-solution}
Let $X$, $\mathcal A$, $T$, $y_0$, and $f$ be as in \Cref{def:abstract-strong-solution}. A function $y\in C([0,T];X)$ is called an
\emph{abstract generalized solution} of $y'+\mathcal Ay\ni f$ with initial datum $y_0$ if there exist $y_0^k\in X$, $f_k\in L^1(0,T;X)$, and corresponding abstract strong solutions $y_k$ of $y_k'+\mathcal Ay_k\ni f_k$, $y_k(0)=y_0^k$,
such that
\[y_0^k\to y_0\quad\text{in }X,\qquad f_k\to f\quad\text{in }L^1(0,T;X),\qquad y_k\to y\quad\text{in }C([0,T];X).\]
\end{defi}
We now specify the notion of strong solution used for the closed-loop equations. The parameters $\gamma_\lambda>0$ and $\mu_\lambda>1$ are fixed in \Cref{sec:sta-Stokes}, where the feedback law \eqref{eq:feedback-Stokes} is constructed. Functions defined on $\omega$ are extended by zero when
multiplied by $\mathds1_\omega$.

\begin{defi}\label{def:strong-solution}
Let $T>0$, $y_0\in V$, and let $d$ satisfy \textup{\ref{A_1}--\ref{A_2}}.
A function $y$ is a \emph{strong solution} of the closed-loop Stokes/Navier--Stokes system {\color{blue}if} $y\in H^1(0,T;H)\cap L^2(0,T;D(S))\cap C([0,T];V)$, $y(0)=y_0$, and there exists a measurable $\eta\in L^\infty(0,T;\mathbf{L^2}(\Omega))$ such that
\[ \eta(t)\in\sign_{\mathbf{L^2}(\Omega)} (\mathds1_\omega\mathscr C_\lambda y(t)) \quad\text{for a.e. }t\in(0,T), \]
and for Stokes system
\[y'+Sy+\gamma_\lambda P(\mathds1_\omega\mathbb P_Ny)+DP(\mathds1_\omega\eta)=P(\mathds1_\omega d)\quad\text{in }L^2(0,T;H),\]
and for Navier--Stokes system
\[y'+Sy+B(y,y)+\gamma_\lambda P(\mathds1_\omega\mathbb P_Ny)+DP(\mathds1_\omega\eta)=P(\mathds1_\omega d)\quad\text{in }L^2(0,T;H).\]
A global strong solution satisfies this definition on every
finite interval $[0,T]$. 
\end{defi}
We recall that uniqueness always refers to the velocity $y$.  Every strong solution in \Cref{def:strong-solution} is also an abstract strong solution of the corresponding projected differential inclusion. For this regularity, the energy identities
\[\frac12\frac{d}{dt}\|y(t)\|_\mu^2=(y'(t),\mathscr C_\lambda y(t))_H, \qquad \frac12\frac{d}{dt}\|y(t)\|_V^2=(y'(t),Sy(t))_H\]
hold almost everywhere, with integrable right-hand sides. In particular, the tests used below with $\mathscr C_\lambda y$ and $Sy$ are justified for strong solutions. For more regular data, we shall additionally obtain $y\in W^{1,\infty}(0,T;H)\cap L^\infty(0,T;D(S))$.\medskip

During our analysis, we shall also consider truncated systems, which are used to construct global solutions in two and three
dimensions under the assumptions of our main results. For $\delta>0$, we define the radial truncation $R_\delta:V\to V$ by
\[R_\delta y:=
\begin{cases}
y, & \|y\|_V\leq\delta,\\[1mm]
\dfrac{\delta}{\|y\|_V}y, & \|y\|_V>\delta.
\end{cases}\]
We then define $B_\delta:V\to V'$ by
\begin{equation}\label{eq:definition_B_delta}
B_\delta(y):=B(R_\delta y,R_\delta y)=
\begin{cases}
B(y,y), & \|y\|_V\leq\delta,\\[1mm]
\left(\dfrac{\delta}{\|y\|_V}\right)^2 B(y,y),
& \|y\|_V>\delta.
\end{cases}
\end{equation}
The truncated closed-loop system is obtained by replacing $B(y,y)$ with $B_\delta(y)$ in the Navier--Stokes equation.\medskip

Finally, we introduce generalized solutions for the Stokes and two-dimensional Navier--Stokes systems by approximation with strong solutions corresponding to regular data.

\begin{defi}\label{def:generalized-solution}
Let $T>0$, $y_0\in H$, and let $d$ satisfy
\textup{\ref{A_1}--\ref{A_2}}.
A function $y\in C([0,T];H)$ is called a \emph{generalized solution} of the closed-loop Stokes system (respectively, the two-dimensional Navier--Stokes system) with data $(y_0,d)$ {\color{blue}if} there exist $y_0^k\in D(S)$, $d^k\in W^{1,1}(0,T;\mathbf{L^2}(\omega))$, with $\|d^k(t)\|_{\mathbf{L^2}(\omega)}\leq D$ for a.e. $t\in(0,T)$, and corresponding strong solutions $y^k$ of the same closed-loop system, with the same feedback law $G_\lambda$, such that
\[y_0^k\to y_0\quad\text{in }H,\qquad d^k\to d\quad\text{in }L^1(0,T;\mathbf{L^2}(\omega)), \qquad y^k\to y\quad\text{in }C([0,T];H).\]
A global generalized solution satisfies this definition
on every finite interval $[0,T]$.
\end{defi}

%-------------
\section{A priori stabilization estimates for the Stokes system}
\label{sec:sta-Stokes}

In this section, we construct the feedback control law and introduce a Lyapunov function designed for stability analysis.
\begin{prop}\label{prop:stabStokes}
Let $\lambda>0$ and assume that $d$ satisfies \textup{\ref{A_1}--\ref{A_2}}. There exists a constant $C_\lambda\geq1$ such that, for every $T>0$, every strong solution $y$ of \eqref{eq:closed_loop_Stokes}, with initial datum $y(0)=y_0$ and the feedback law $G_\lambda$ defined in \eqref{eq:feedback-Stokes}, satisfies the following a priori estimate:
\begin{equation}\label{eq:fin_Stabilization-Stokes}
\|y(t)\|_H\leq C_\lambda e^{-\frac{\lambda}{4}t}\|y_0\|_H,
\qquad t\in[0,T].
\end{equation}
\end{prop}
\begin{proof}
Fix $T>0$ and let $y$ be a strong solution of \eqref{eq:closed_loop_Stokes}. Adapting the methodology proposed by Xiang in \cite{xiang2023small}, we define the feedback law by
\begin{equation*}
\begin{aligned}
G_\lambda (y(t)):= -\gamma_\lambda \mathbb{P}_N y(t)+\hat{u}, \quad  y(t)\in H, \  t\in [0,\infty), 
\end{aligned}
\end{equation*}
where
\[\gamma_\lambda:=\lambda C_1e^{C_1\sqrt{\lambda}} \quad \text{ and } \quad \mu_\lambda:=(\gamma_\lambda/\lambda)^2=C_1^2e^{2C_1\sqrt{\lambda}}>1,\]
and  $\hat u$ will be defined later. The feedback law is decomposed in two terms; the first one $-\gamma_\lambda\mathbb P_Ny$ is taken from \cite{xiang2023small} and provides rapid stabilization, while $\hat u$ is designed to reject the disturbance. \noindent Let us introduce the Lyapunov functional $V: H\longrightarrow \mathbb{R}_0^+$ defined by
\begin{equation}\label{eq:functional}
V(y(t)):=\mu_{\lambda} \|\mathbb{P}_N y\|_{H}^2 +  \| \mathbb{P}_N^\perp y \|_{H}^2.
\end{equation}
Notice that
\begin{equation}\label{eq:equiv_V-H}
V(y)=\|y\|_\mu^2, \qquad \|y\|_{H}^2 \leq V(y) \leq \mu_\lambda\|y\|_{H}^2 ,\quad \forall y \in H.
\end{equation}
Since $y\in H^1(0,T;H)$, the following identity holds for almost every $t\in(0,T)$:
\begin{equation*}
\begin{aligned}  
\frac{d}{dt} V(y(t)) = \mu_\lambda \frac{d}{dt} \left( \mathbb{P}_N y,  \mathbb{P}_N y \right)_{H} + \frac{d}{dt}\left( \mathbb{P}_N^\perp y,  \mathbb{P}_N^\perp y \right)_{H} = 2\mu_\lambda \left(  \mathbb{P}_N y,  \frac{d}{dt}y \right)_{H}  +  2\left( \mathbb{P}_N^\perp y, \frac{d}{dt}y\right)_{H}. 
\end{aligned}
\end{equation*} 
We estimate both terms using the projected equation in $H$. Since $y\in L^2(0,T;D(S))$, all the following identities hold for almost every $t\in(0,T)$. 
\begin{equation*}
\begin{aligned}
\left(\mathbb P_Ny,\frac{d}{dt}y\right)_H &= \left(\mathbb P_Ny,-Sy-\gamma_\lambda P(\mathds1_\omega\mathbb P_Ny)+P(\mathds1_\omega(\hat u+d))\right)_H\\
&=-\|\mathbb P_Ny\|_V^2 -\gamma_\lambda\left(\mathbb P_Ny,\mathds1_\omega\mathbb P_Ny\right)_{\mathbf{L^2}(\Omega)}+\left(\mathbb P_Ny,\mathds1_\omega(\hat u+d)\right)_{\mathbf{L^2}(\Omega)}\\
&=-\|\mathbb P_Ny\|_V^2 -\gamma_\lambda\|\mathbb P_Ny\|_{\mathbf{L^2}(\omega)}^2+\left(\mathbb P_Ny,\mathds1_\omega(\hat u+d)\right)_{\mathbf{L^2}(\Omega)}.
\end{aligned}
\end{equation*}
Furthermore, using the spectral inequality \eqref{lebeau_chaves}, we obtain
\begin{equation*}
\begin{aligned}
\left(\mathbb P_Ny,\frac{d}{dt}y\right)_H
&\leq-\|\mathbb P_Ny\|_V^2
-\gamma_\lambda\left(C_1e^{C_1\sqrt{\lambda}}\right)^{-1}
\|\mathbb P_Ny\|_H^2
+\left(\mathds1_\omega\mathbb P_Ny,\hat u+d\right)_{\mathbf{L^2}(\Omega)}\\
&=-\|\mathbb P_Ny\|_V^2
-\lambda\|\mathbb P_Ny\|_H^2
+\left(\mathds1_\omega\mathbb P_Ny,\hat u+d\right)_{\mathbf{L^2}(\Omega)}.
\end{aligned}
\end{equation*}
On the other hand, we have
\begin{equation*}
\begin{aligned}
\left(\mathbb P_N^\perp y,\frac{d}{dt}y\right)_H
&= \left(\mathbb P_N^\perp y,-Sy-\gamma_\lambda P(\mathds1_\omega\mathbb P_Ny)+P(\mathds1_\omega(\hat u+d))\right)_H\\
&=-\|\mathbb P_N^\perp y\|_V^2
-\gamma_\lambda\left(\mathbb P_N^\perp y,\mathds1_\omega\mathbb P_Ny\right)_{\mathbf{L^2}(\Omega)}
+\left(\mathbb P_N^\perp y,\mathds1_\omega(\hat u+d)\right)_{\mathbf{L^2}(\Omega)}\\
&\leq-\|\mathbb P_N^\perp y\|_V^2
+\gamma_\lambda\|\mathbb P_N^\perp y\|_H
\|\mathbb P_Ny\|_{\mathbf{L^2}(\omega)}
+\left(\mathbb P_N^\perp y,\mathds1_\omega(\hat u+d)\right)_{\mathbf{L^2}(\Omega)}.
\end{aligned}
\end{equation*}
Since $\tau_j>\lambda$ for every $j>N$, we have
$\lambda\|\mathbb P_N^\perp y\|_H^2
\leq\|\mathbb P_N^\perp y\|_V^2$.
Using Young's inequality, we obtain
\[\gamma_\lambda\|\mathbb P_N^\perp y\|_H
\|\mathbb P_Ny\|_{\mathbf{L^2}(\omega)}
\leq\frac12\|\mathbb P_N^\perp y\|_V^2
+\frac{\gamma_\lambda^2}{2\lambda}\|\mathbb P_Ny\|_H^2.\]
Therefore,
\begin{equation*}
\left(\mathbb P_N^\perp y,\frac{d}{dt}y\right)_H
\leq-\frac12\|\mathbb P_N^\perp y\|_V^2
+\frac{\gamma_\lambda^2}{2\lambda}\|\mathbb P_Ny\|_H^2
+\left(\mathbb P_N^\perp y,\mathds1_\omega(\hat u+d)\right)_{\mathbf{L^2}(\Omega)}.
\end{equation*}
Gathering terms and using again
$\lambda\|\mathbb P_N^\perp y\|_H^2\leq\|\mathbb P_N^\perp y\|_V^2$,
we obtain
\begin{equation*}
\begin{aligned}
\frac{d}{dt}V(y(t))
\leq{}&-2\mu_\lambda\|\mathbb P_Ny\|_V^2
-\frac12\|\mathbb P_N^\perp y\|_V^2
-\frac{\lambda}{2}\|\mathbb P_N^\perp y\|_H^2
-2\mu_\lambda\lambda\|\mathbb P_Ny\|_H^2\\
&+\frac{\gamma_\lambda^2}{\lambda}\|\mathbb P_Ny\|_H^2
+2\bigl(\mathscr C_\lambda y,\mathds1_\omega(\hat u+d)\bigr)_{\mathbf{L^2}(\Omega)}.
\end{aligned}
\end{equation*}
Since $\mu_\lambda=(\gamma_\lambda/\lambda)^2$, we have
$\gamma_\lambda^2/\lambda=\mu_\lambda\lambda$. Therefore,
\begin{equation*}
\begin{aligned}
\frac{d}{dt}V(y(t))
\leq{}&-2\mu_\lambda\|\mathbb P_Ny\|_V^2
-\frac12\|\mathbb P_N^\perp y\|_V^2
-\mu_\lambda\lambda\|\mathbb P_Ny\|_H^2
-\frac{\lambda}{2}\|\mathbb P_N^\perp y\|_H^2+2\bigl(\mathscr C_\lambda y,\mathds1_\omega(\hat u+d)\bigr)_{\mathbf{L^2}(\Omega)}.
\end{aligned}
\end{equation*}
Since $\mu_\lambda>1$, this gives
\begin{equation}\label{eq:stab}
\frac{d}{dt}V(y(t))
\leq-\frac12\|y\|_V^2-\frac{\lambda}{2}V(y(t))
+2\bigl(\mathscr C_\lambda y,\mathds1_\omega(\hat u+d)\bigr)_{\mathbf{L^2}(\Omega)}.
\end{equation}

%-----------------
\noindent \textbf{Design of \texorpdfstring{$\hat{u}$.}{u}}\label{subsec2-2}
To attenuate the effect of the unknown disturbance $d$, we now construct the control term $ \hat{u}$, following the recent strategy proposed in \cite{guzman2025rapid}. Since the disturbance is supposed to be unknown, we can not choose $\hat{u}=-d$. Under assumption \ref{A_2}, $ \|d\|_{\mathbf{L^2}(\omega)} \leq D $. Then, it suffices to consider the feedback law
\begin{equation*}
\hat{u}=-D \sign_{\mathbf{L^2}(\Omega)}\left(\mathds 1_\omega \mathscr{C}_\lambda y\right).
\end{equation*}
Thus, the feedback law takes the form
\begin{equation}\label{eq:feedback-Stokes}
G_\lambda(y)=-\gamma_\lambda\mathbb P_Ny-D\sign_{\mathbf{L^2}(\Omega)}\left(\mathds1_\omega\mathscr C_\lambda y\right),\quad y\in H.
\end{equation}
Coming back to \eqref{eq:stab}, we estimate the third term using the Cauchy-Schwarz inequality, as follows
\begin{equation*}
\begin{aligned}
\left(\hat{u}+d,\mathds 1_\omega \mathscr{C}_\lambda y\right)_{\mathbf{L^2}(\Omega)} &= \left(\hat{u},\mathds 1_\omega \mathscr{C}_\lambda y\right)_{\mathbf{L^2}(\Omega)}+\left(d,\mathds 1_\omega \mathscr{C}_\lambda y\right)_{\mathbf{L^2}(\Omega)}\\
&\leq \left(\hat{u},\mathds 1_\omega \mathscr{C}_\lambda y\right)_{\mathbf{L^2}(\Omega)}+\left\|d\right\|_{\mathbf{L^2}(\omega)}\left\|\mathds 1_\omega \mathscr{C}_\lambda  y\right\|_{\mathbf{L^2}(\Omega)}.
\end{aligned}
\end{equation*}
Let $\theta\in \sign_{\mathbf{L^2}(\Omega)} \left(\mathds 1_\omega \mathscr{C}_\lambda y\right)$, it follows that
\begin{equation}\label{eq:ineq_sign}
\begin{aligned}
\left(\hat{u}+d,\mathds 1_\omega \mathscr{C}_\lambda y\right)_{\mathbf{L^2}(\Omega)}&\leq -D\left(\theta,\mathds 1_\omega \mathscr{C}_\lambda y\right)_{\mathbf{L^2}(\Omega)} +D\left\|\mathds 1_\omega \mathscr{C}_\lambda y\right\|_{\mathbf{L^2}(\Omega)}\\
&\leq -D\left\|\mathds 1_\omega \mathscr{C}_\lambda y\right\|_{\mathbf{L^2}(\Omega)}+D\left\|\mathds 1_\omega \mathscr{C}_\lambda y\right\|_{\mathbf{L^2}(\Omega)}=0.
\end{aligned}
\end{equation}
Therefore,
\begin{equation}\label{eq:stabilization-Stokes}
\frac{d}{dt}V(y(t))\le-\frac12\| y\|_{V}^2-\frac\lambda2 V(y(t)).
\end{equation}
This differential inequality implies the exponential decay of $V(y(t))$. Indeed, we observe that
\begin{equation*}
V(y(t))\leq e^{-\frac{\lambda}{2}t}V(y(0)), \qquad t\in[0,T].
\end{equation*}
Consequently,
\begin{equation*}
\|y(t)\|_{H}^2 \leq V(y(t)) \leq e^{-\frac{\lambda}{2} t} V(y(0)) \leq e^{-\frac{\lambda}{2} t} \mu_\lambda\|y(0)\|_{H}^2  \leq C_1^2 e^{2 C_1 \sqrt{\lambda}} e^{-\frac{\lambda}{2} t}\|y(0)\|_{H}^2 .   
\end{equation*}
Thus, we obtain \eqref{eq:fin_Stabilization-Stokes} with $C_\lambda =\sqrt{\mu_\lambda}=C_1e^{C_1\sqrt{\lambda}}.$
\end{proof}

%------------------------
\section{Well-posedness of the closed-loop Stokes system}
\label{sec:WP-Stokes}
In this section, we establish the existence and uniqueness
of generalized solutions to the closed-loop Stokes system using maximal monotone operator theory. For regular data, we obtain strong solutions with sufficient regularity to apply \Cref{prop:stabStokes}. We then extend the exponential
decay estimate to generalized solutions by approximation
and continuous dependence on the data. Consider the closed-loop system
\begin{equation}\label{eq:closed-loop-stokes}
\begin{cases}
\partial_t y-\Delta y+\nabla p + \gamma_\lambda\mathds 1_\omega \mathbb P_N y+D\mathds 1_\omega\eta =
\mathds 1_\omega d, & \text{in } \Omega\times(0,T),\\
\nabla\cdot y =0, & \text{in }\Omega\times(0,T),\\
\eta\in\sign_{\mathbf{L^2}(\Omega)} (\mathds 1_\omega \mathscr{C}_\lambda y), & \text{a.e. in }(0,T),\\
y=0, & \text{on }\partial\Omega\times(0,T),\\
y(0)=y_0,& \text{in }\Omega.
\end{cases}
\end{equation}
Applying the Leray projection $P$ to \eqref{eq:closed-loop-stokes}, we obtain the abstract differential inclusion
\begin{equation}\label{eq:abstract-stokes}
\begin{cases}
y'+A y\ni P(\mathds 1_\omega d), & \text{a.e. in } (0,T).\\
y(0)=y_0,
\end{cases}
\end{equation}
Here, $P(\partial_t y)=\partial_t(Py)=\partial_t y$ in the sense of distributions in time, since $Py=y$. The multivalued operator $A$ is given by
\begin{equation}\label{eq:A-def}
\begin{array}{rcl}
A:D(S) &\longrightarrow& 2^H,\\
y &\longmapsto& Ay:=L_\lambda y+\mathcal By.
\end{array}
\end{equation}
where
\begin{equation}\label{eq:Llambda-def}
\begin{aligned}
L_\lambda y
&:=Sy+\gamma_\lambda P(\mathds1_\omega\mathbb P_Ny),
&&y\in D(S),\\
\mathcal By
&:=DP\left(\mathds1_\omega
\sign_{\mathbf{L^2}(\Omega)}
\left(\mathds1_\omega\mathscr C_\lambda y\right)\right),
&&y\in H.
\end{aligned}
\end{equation}
Observe that the differential inclusion \eqref{eq:abstract-stokes} is equivalent to the following equation,
\[y'+Sy+\gamma_\lambda P(\mathds 1_\omega\mathbb P_N y)+DP(\mathds 1_\omega\eta)=P(\mathds 1_\omega d),\]
with $\eta\in \sign_{\mathbf{L^2}(\Omega)} (\mathds 1_\omega \mathscr C_\lambda y)$. Note that pressure $p$ disappears after applying $P$ since $P(\nabla p)=0$.\medskip

To prove the well-posedness of system \eqref{eq:closed-loop-stokes}, we will use maximal monotone operator theory. More precisely, the argument is based on two fundamental properties: the monotonicity of $A$ and the surjectivity of $I+A$, where the inner product $(\cdot,\cdot)_\mu$ plays a crucial role. Let $H_\mu$ denote the space $H$ endowed with the inner product $(\cdot,\cdot)_\mu$. Since $\mathscr{C}_\lambda$ is an isomorphism from $H$ onto $H$, the norms induced by $(\cdot,\cdot)_H$ and $(\cdot,\cdot)_\mu$ are equivalent, defining the same topology. Together, these properties allow us to apply the classical existence and uniqueness theory.

\begin{prop}\label{prop:A-monotone-stokes}
The operator $A$ defined by \eqref{eq:A-def} is monotone in $H_\mu$.
\end{prop}
\begin{proof}
Let $y_i\in D(A)$ and let $\xi_i\in Ay_i$, $i=1,2$. By \eqref{eq:A-def}, there exist $\eta_i\in \sign_{\mathbf{L^2}(\Omega)} (\mathds 1_\omega \mathscr C_\lambda y_i)$ such that $\xi_i=L_\lambda y_i+DP(\mathds 1_\omega\eta_i)$. Set $z=y_1-y_2$, then
\begin{equation}\label{eq:monotone-split}
(\xi_1-\xi_2,z)_\mu=(L_\lambda z,z)_\mu+D(P(\mathds 1_\omega(\eta_1-\eta_2)),z)_\mu.
\end{equation}
%The pressure does not appear in \eqref{eq:monotone-split}, since the equation has been projected by $P$. 
%Since $\mathscr{C}_\lambda z\in H$, one has  $(\nabla(p_1-p_2),\mathscr C_\lambda z)_{H}=0.$
Using the spectral inequality \eqref{lebeau_chaves} and Young's
inequality as in Section~\ref{sec:sta-Stokes}, we obtain
\begin{equation}\label{eq:Llambda-coercive}
\begin{aligned}
(L_\lambda z,z)_\mu
&=\mu_\lambda\|\mathbb P_Nz\|_V^2
+\|\mathbb P_N^\perp z\|_V^2
+\mu_\lambda\gamma_\lambda\|\mathbb P_Nz\|_{\mathbf{L^2}(\omega)}^2
+\gamma_\lambda(\mathbb P_Nz,\mathbb P_N^\perp z)_{\mathbf{L^2}(\omega)}\\
&\geq\mu_\lambda\|\mathbb P_Nz\|_V^2
+\frac12\|\mathbb P_N^\perp z\|_V^2
+\left(\mu_\lambda\lambda-\frac{\gamma_\lambda^2}{2\lambda}\right)
\|\mathbb P_Nz\|_H^2\\
&\geq\frac12\|z\|_V^2.
\end{aligned}
\end{equation}
Here we used $\mu_\lambda=(\gamma_\lambda/\lambda)^2>1$. It remains to treat the sign term. Using \eqref{eq:mu-C},
the orthogonality of $P$ and the fact that $\mathscr C_\lambda z\in H$, we obtain
\[\begin{aligned}
D(P(\mathds1_\omega(\eta_1-\eta_2)),z)_\mu
&=D(\mathds1_\omega(\eta_1-\eta_2),\mathscr C_\lambda z)_{\mathbf{L^2}(\Omega)}\\
&=D(\eta_1-\eta_2,
\mathds1_\omega\mathscr C_\lambda y_1
-\mathds1_\omega\mathscr C_\lambda y_2)_{\mathbf{L^2}(\Omega)}.
\end{aligned}\]
Since $\sign_{\mathbf{L^2}(\Omega)}$ is monotone, it follows that
\begin{equation}\label{eq:sign-monotone}
D(P(\mathds 1_\omega(\eta_1-\eta_2)),z)_\mu\geq0.
\end{equation}
Combining \eqref{eq:monotone-split}, \eqref{eq:Llambda-coercive}, and
\eqref{eq:sign-monotone}, we get
\[(\xi_1-\xi_2,y_1-y_2)_\mu\geq \frac 12\|y_1-y_2\|_V^2\geq0.\]
Therefore, $A$ is monotone in $H_\mu$.
\end{proof}

To prove the maximal monotonicity of $A$, a key point is to identify  $\mathcal{B}$ with the subdifferential in $H_\mu$ of a convex functional.

\begin{prop}\label{prop:A-maximal-stokes}
The operator $A$ defined by \eqref{eq:A-def} is maximal monotone in $H_\mu$.
\end{prop}
\begin{proof}
We prove this by writing $A$ as the sum of a maximal monotone linear operator and a subdifferential.  Let us consider
\begin{equation}
\begin{array}{rcl}
\Phi:H_\mu & \longrightarrow&\mathbb R,\\
y & \longmapsto & \Phi(y):=  D\|\mathds 1_\omega \mathscr C_\lambda y\|_{\mathbf{L^2}(\Omega)}. 
\end{array}
\end{equation}
Since $\mathds 1_\omega \mathscr C_\lambda \in \mathcal L (H_\mu; \mathbf{L^2}(\Omega))$ and the norm of $\mathbf{L^2}(\Omega)$ is convex and continuous, $\Phi$ is proper, convex and lower semicontinuous on $H_\mu$. Hence, by \cite[Theorem~A, p.~210]{Rockafellar1970}, the subdifferential $\partial_\mu\Phi$,  taken with respect to the inner product $(\cdot,\cdot)_\mu$, is maximal monotone in $H_\mu$. We now characterize the subdifferential $\partial_\mu\Phi.$
\begin{comment}
Let $y\in H$, by definition, $\xi\in\partial_\mu\Phi(y)$
if and only if
\[\Phi(y+h)-\Phi(y)\geq (\xi,h)_\mu,\quad \forall\,h\in H.\]
Let $\eta\in\sign_{\mathbf L^2(\Omega)}\bigl(\mathds{1}_\omega\mathscr C_\lambda y\bigr).$ Since $\sign_{\mathbf L^2(\Omega)}=\partial\|\cdot\|_{\mathbf L^2(\Omega)},$ for every $h\in H$,
\[\begin{aligned}
 \Phi(y+h)-\Phi(y)&=D\|\mathds{1}_\omega\mathscr C_\lambda(y+h)\|_{\mathbf L^2(\Omega)}-
D\|\mathds{1}_\omega\mathscr C_\lambda y\|_{\mathbf L^2(\Omega)}
\\
&\geq D\bigl(\eta,\mathds{1}_\omega\mathscr C_\lambda h
\bigr)_{\mathbf L^2(\Omega)}.   
\end{aligned}\]
Since $\mathscr C_\lambda h\in H$ and $P$ is the orthogonal projection
onto $H$,
\[D\bigl(\eta,\mathds{1}_\omega\mathscr C_\lambda h\bigr)_{\mathbf L^2(\Omega)}=D\bigl(\mathds{1}_\omega\eta,\mathscr C_\lambda h\bigr)_{\mathbf L^2(\Omega)}=D\bigl(P(\mathds{1}_\omega\eta),\mathscr C_\lambda h\bigr)_H=\bigl(DP(\mathds{1}_\omega\eta),h\bigr)_\mu.\]
Therefore, $DP ( \mathds{1}_\omega \eta) \in \partial_\mu \Phi(y),$
and hence
\[\left\{DP(\mathds{1}_\omega\eta):\eta\in\sign_{\mathbf L^2(\Omega)}\bigl(\mathds{1}_\omega\mathscr C_\lambda y\bigr)\right\}\subset\partial_\mu\Phi(y).\]
Conversely, 
\end{comment}
Set $L:=\mathds1_\omega\mathscr C_\lambda$, so that
$\Phi(y)=D\|Ly\|_{\mathbf{L^2}(\Omega)}$. Since $L\in\mathcal L(H_\mu;\mathbf{L^2}(\Omega))$ and the norm is continuous on $\mathbf{L^2}(\Omega)$, the subdifferential chain rule yields
\[\partial_\mu\Phi(y)
=DL_\mu^*\sign_{\mathbf{L^2}(\Omega)}(Ly),\]
where $L_\mu^*$ denotes the adjoint with respect to the inner product $(\cdot,\cdot)_\mu$
characterized by
\[(L_\mu^*\eta,h)_\mu=(\eta,Lh)_{\mathbf L^2(\Omega)},\qquad\eta\in\mathbf L^2(\Omega),\ h\in H.\]
Indeed,
\[(\eta,Lh)_{\mathbf L^2(\Omega)}=(\mathds{1}_\omega\eta,\mathscr C_\lambda h)_{\mathbf L^2(\Omega)}=\bigl(P(\mathds{1}_\omega\eta),\mathscr C_\lambda h\bigr)_H=\bigl(P(\mathds{1}_\omega\eta),h\bigr)_\mu.\]
Hence, $L_\mu^*\eta=P(\mathds{1}_\omega\eta),$ that is, $L_\mu^* = P \mathds{1}_\omega.$ Consequently,
\[\partial_\mu\Phi(y)=\left\{DP(\mathds{1}_\omega\eta):\eta\in\sign_{\mathbf L^2(\Omega)}\bigl(\mathds{1}_\omega\mathscr C_\lambda y\bigr)\right\}.\]
The above characterization shows that
\begin{equation}\label{eq:B-subdiff}
\mathcal B=\partial_\mu\Phi.
\end{equation}
Consider the linear operator $L_\lambda$ defined in \eqref{eq:Llambda-def}. Its monotonicity follows from
\eqref{eq:Llambda-coercive}. For a given $f\in H$, to prove maximality we solve
\[y+L_\lambda y=f.\]
We seek $y\in V$ satisfying
\begin{equation}\label{eq:linear-resolvent-var}
(y,z)_\mu+(y,\mathscr C_\lambda z)_V
+\gamma_\lambda
(\mathds1_\omega\mathbb P_Ny,\mathscr C_\lambda z)_{\mathbf{L^2}(\Omega)}
=(f,z)_\mu,\quad  z\in V.
\end{equation}
By \eqref{eq:Llambda-coercive} and the density of $D(S)$ in $V$,
we have, for every $y\in V$,
\[\|y\|_\mu^2+(y,\mathscr C_\lambda y)_V+\gamma_\lambda(\mathds1_\omega\mathbb P_Ny,\mathscr C_\lambda y)_{\mathbf{L^2}(\Omega)}\geq \|y\|_\mu^2+\frac12\|y\|_V^2.\]
Therefore, the bilinear form in \eqref{eq:linear-resolvent-var}
is continuous and coercive on $V$. By the Lax--Milgram theorem,
there exists a unique $y\in V$ satisfying \eqref{eq:linear-resolvent-var}. Since $\mathscr C_\lambda^{-1}
=\mu_\lambda^{-1}\mathbb P_N+\mathbb P_N^\perp$ preserves $V$,
taking $z=\mathscr C_\lambda^{-1}v$ in \eqref{eq:linear-resolvent-var} gives
\[(y,v)_V=\bigl(f-y-\gamma_\lambda P(\mathds1_\omega\mathbb P_Ny),v\bigr)_H,\quad  v\in V.\]
Since $f-y-\gamma_\lambda P(\mathds1_\omega\mathbb P_Ny)\in H$,
the variational characterization of the Stokes operator
\cite[Chapter~III, Lemma~3.7]{temam84}
implies that $y\in D(S)$ and
\[Sy=f-y-\gamma_\lambda P(\mathds1_\omega\mathbb P_Ny).\]
Thus, $y+L_\lambda y=f$, proving that $R(I+L_\lambda)=H$.
Together with its monotonicity, this shows that $L_\lambda$ is maximal monotone in $H_\mu$. Finally, since $D(\mathcal B)=D(\partial_\mu\Phi)=H_\mu$ we have $D(L_\lambda)\cap\operatorname{int}D(\partial_\mu\Phi)
=D(S)\neq\emptyset$. Therefore, the sum theorem for maximal monotone operators \cite[Corollary~2.7]{Brezis1973} implies that $A=L_\lambda+\partial_\mu\Phi$ is maximal monotone in $H_\mu$.
\end{proof}

\begin{rem}
The use of the weighted Hilbert structure $H_\mu$ is essential here. If we take the subdifferential of $\Phi$ with respect to the usual inner product of $H$, then the chain rule yields
\[\partial_H \Phi(y)=\left\{ D \mathscr{C}_\lambda^*P(\mathds 1_\omega \eta) : \eta\in\sign_{\mathbf L^2(\Omega)}(\mathds 1_\omega \mathscr{C}_\lambda y) \right\},\]
where $\mathscr{C}_\lambda^*$ denotes the adjoint of $\mathscr{C}_\lambda$ in $H$. This expression introduces an additional operator and does not match the simple feedback structure $DP(\mathds 1_\omega \eta)$ obtained in \eqref{eq:B-subdiff}. Using this argument, we simplify the well-posedness proof of \cite[Prop~5.2]{guzman2025rapid}.
\end{rem}
Whenever $v$ and $\xi$ are strongly measurable and $\xi(t)\in\mathcal Bv(t)$, the characterization \eqref{eq:B-subdiff} admits a measurable sign selection. Indeed, the sets
\[\{\eta\in\mathbf{L^2}(\Omega):\ \eta\in\sign_{\mathbf{L^2}(\Omega)}(\mathds1_\omega\mathscr C_\lambda v(t)),\quad DP(\mathds1_\omega\eta)=\xi(t)\}
\]
are nonempty, closed, and contained in the unit ball, with a measurable graph. The measurable selection theorem \cite[Chapter~8]{aubinfrankowska1990} in the separable Hilbert space $\mathbf{L^2}(\Omega)$ therefore gives a Lebesgue-measurable choice $\eta$, with $\|\eta(t)\|_{\mathbf{L^2}(\Omega)}\leq1$.\medskip

We can now state our well-posedness result for the closed-loop Stokes system.

\begin{prop}\label{thm:well-posedness-stokes}
Let $T>0$ and assume that $d$ satisfies
\textup{\ref{A_1}--\ref{A_2}}.
For every $y_0\in H$, the closed-loop Stokes system \eqref{eq:closed-loop-stokes} admits a unique abstract generalized
solution $y\in C([0,T];H)$. If, in addition, $y_0\in D(S)$ and
$d\in W^{1,1}(0,T;\mathbf{L^2}(\omega))$, then $y$ is a strong solution satisfying
\[y\in W^{1,\infty}(0,T;H)\cap L^\infty(0,T;D(S)).\]
Moreover, any two abstract generalized solutions $y,\widehat y$ corresponding to data $(y_0,d)$ and $(\widehat y_0,\widehat d)$ satisfy
\begin{equation}\label{eq:Cauchy_Showalter}
\|y(t)-\widehat y(t)\|_\mu \leq \|y_0-\widehat y_0\|_\mu
+\sqrt{\mu_\lambda}\int_0^t \|d(s)-\widehat d(s)\|_{\mathbf{L^2}(\omega)}\,ds, \qquad 0\leq t\leq T.
\end{equation}
\end{prop}
\begin{proof}
Let $T>0$. By \Cref{prop:A-maximal-stokes}, $A$ is maximal monotone in $H_\mu$. Moreover, $D(A)=D(S)$ is dense in $H_\mu$, and assumption \ref{A_2} gives $f=P(\mathds1_\omega d)\in L^1(0,T;H_\mu)$. Therefore, \cite[Chapter~IV, Theorem~4.1a, p.~183]{s1997} implies that, for every $y_0\in H$, the abstract problem \eqref{eq:abstract-stokes} admits a unique abstract generalized solution $y\in C([0,T];H_\mu)$. By the equivalence of the norms $\|\cdot\|_\mu$ and $\|\cdot\|_H$, we also have $y\in C([0,T];H)$. Assume now, in addition, that $d\in W^{1,1}(0,T;\mathbf{L^2}(\omega))$ and $y_0\in D(A)$. Then, \cite[Chapter~IV, Theorem~4.1]{s1997} gives the existence of a unique abstract strong solution $y\in W^{1,\infty}(0,T;H)$ of \eqref{eq:abstract-stokes}, satisfying $y(t)\in D(A)$ for every $t\in[0,T]$. For a.e. $t\in(0,T)$, the inclusion \eqref{eq:abstract-stokes} gives
\[Sy(t)=f(t)-y'(t)-\gamma_\lambda P(\mathds1_\omega\mathbb P_Ny(t))-b(t),
\qquad b(t)\in\mathcal By(t).\]
By definition, $\|b(t)\|_H\leq D$. Since $f,y',y\in L^\infty(0,T;H)$, we obtain $Sy\in L^\infty(0,T;H)$. Thus, both $y$ and $Sy$ are bounded in $H$, which gives $y\in L^\infty(0,T;D(S))$ for the graph norm $\|v\|_{D(S)}:=\|v\|_H+\|Sv\|_H$. The embedding $H^1(0,T;H)\cap L^2(0,T;D(S))\hookrightarrow C([0,T];V)$ gives $y\in C([0,T];V)$. The sign term admits the measurable representation described above, so the solution is strong in the sense of \Cref{def:strong-solution}.
Finally, let $y,\widehat y$ be the strong solutions of \eqref{eq:abstract-stokes} corresponding to disturbances
$d,\widehat d\in W^{1,1}(0,T;\mathbf{L^2}(\omega))$
and initial data $y_0,\widehat y_0\in D(A)$. By the monotonicity of $A$ in $H_\mu$, we obtain
\[\begin{aligned}
\frac12\frac{d}{dt}\|y(t)-\widehat y(t)\|_\mu^2&\leq
\bigl(P(\mathds1_\omega(d(t)-\widehat d(t))),y(t)-\widehat y(t)\bigr)_\mu\\
&\leq \sqrt{\mu_\lambda}\|d(t)-\widehat d(t)\|_{\mathbf{L^2}(\omega)}\|y(t)-\widehat y(t)\|_\mu.
\end{aligned}\]
Dividing by $(\|y(t)-\widehat y(t)\|_\mu^2+\varepsilon^2)^{1/2}$, integrating in time, and letting $\varepsilon\to0$ yields \eqref{eq:Cauchy_Showalter} for strong solutions. The same argument, applied first on $[\tau,t]$ and then
letting $\tau\to0$, gives
\[\|y(t)-\widehat y(t)\|_\mu\leq \|y_0-\widehat y_0\|_\mu+\int_0^t\|f(s)-\widehat f(s)\|_\mu ds\]
for any two abstract strong solutions with right-hand sides
$f,\widehat f\in L^1(0,T;H_\mu)$.
By \Cref{def:abstract-generalized-solution}, this estimate
passes to abstract generalized solutions.
Taking $f=P(\mathds1_\omega d)$ and
$\widehat f=P(\mathds1_\omega\widehat d)$, and using
\[\|P(\mathds1_\omega(d-\widehat d))\|_\mu\leq\sqrt{\mu_\lambda}\|d-\widehat d\|_{\mathbf{L^2}(\omega)},\]
we obtain \eqref{eq:Cauchy_Showalter}.
\end{proof}

Note that the arguments presented in \Cref{sec:sta-Stokes}
apply to the strong solutions obtained above for regular data.
We therefore obtain the following corollary.
\begin{cor}\label{cor:stab-Stokes-strong}
Let $\lambda>0$, $T>0$, and $y_0\in D(S)$. Assume that $d$ satisfies \textup{\ref{A_1}--\ref{A_2}} and $d\in W^{1,1}(0,T;\mathbf{L^2}(\omega))$. Then the unique strong solution of \eqref{eq:abstract-stokes} satisfies \eqref{eq:fin_Stabilization-Stokes} for every $t\in[0,T]$.
\end{cor}
\begin{proof}
By \Cref{thm:well-posedness-stokes}, the solution satisfies
$y\in W^{1,\infty}(0,T;H)\cap L^\infty(0,T;D(S))$. Therefore, \Cref{prop:stabStokes} applies.
\end{proof}

We now extend the exponential decay estimate to generalized
solutions and show that they belong to $L^2(0,T;V)$ for every
$T>0$. This completes the proof of our main result for the
Stokes system.

\medskip
\noindent\emph{Proof of \Cref{thm:sta-Stokes}.}
We use the feedback law $G_\lambda$ defined in \eqref{eq:feedback-Stokes}. Fix $T>0$ and $y_0\in H$.
By \Cref{thm:well-posedness-stokes}, there exists a unique abstract
generalized solution $y\in C([0,T];H)$ of \eqref{eq:abstract-stokes}. Choose $y_0^k\in D(S)$ and $d^k\in W^{1,1}(0,T;\mathbf{L^2}(\omega))$ such that
\[y_0^k\to y_0\quad\text{in }H,\qquad d^k\to d\quad\text{in }L^1(0,T;\mathbf{L^2}(\omega)),\]
with the same disturbance bound
\[\|d^k(t)\|_{\mathbf{L^2}(\omega)}\leq D, \qquad\text{for a.e. }t\in(0,T).\]
Let $y^k$ be the strong solution corresponding to
$(y_0^k,d^k)$. By \eqref{eq:Cauchy_Showalter},
\[\sup_{t\in[0,T]}\|y^k(t)-y(t)\|_\mu\leq \|y_0^k-y_0\|_\mu+\sqrt{\mu_\lambda}\int_0^T \|d^k(s)-d(s)\|_{\mathbf{L^2}(\omega)} ds \xrightarrow[k\to\infty]{}0.\]
In particular, $y^k\to y$ in $C([0,T];H)$. By \Cref{def:generalized-solution}, the approximation above shows that $y$ is a generalized solution of the closed-loop Stokes system. Moreover, every generalized solution in this
sense is an abstract generalized solution of \eqref{eq:abstract-stokes}. Its uniqueness therefore follows from \Cref{thm:well-posedness-stokes}. We first establish the energy regularity of $y$. Integrating \eqref{eq:stabilization-Stokes} along $y^k$ gives
\begin{equation}\label{eq:energy-estimate}
\|y^k(t)\|_\mu^2
+\frac12\int_0^t\|y^k(s)\|_V^2\,ds
\leq \|y_0^k\|_\mu^2,
\qquad t\in[0,T].
\end{equation}
Thus, $\{y^k\}$ is bounded in $L^2(0,T;V)$. Its convergence in $C([0,T];H)$ identifies every weak limit in $L^2(0,T;V)$ with $y$. Consequently, $y\in L^2(0,T;V)$, and weak lower semicontinuity yields
\[\|y(t)\|_\mu^2+\frac12\int_0^t\|y(s)\|_V^2\,ds \leq \|y_0\|_\mu^2,\qquad t\in[0,T].\]
We now identify the feedback term in the limit. For each $k$, choose a measurable selection $\eta^k$ such that
\[\eta^k(t)\in\sign_{\mathbf{L^2}(\Omega)}(\mathds1_\omega\mathscr C_\lambda y^k(t)),\qquad\|\eta^k(t)\|_{\mathbf{L^2}(\Omega)}\leq1,\]
and
\[(y^k)'+L_\lambda y^k+DP(\mathds1_\omega\eta^k)=P(\mathds1_\omega d^k).\]
Up to a subsequence,
\[\eta^k\rightharpoonup^\ast\eta \quad\text{in }L^\infty(0,T;\mathbf{L^2}(\Omega)),\qquad \|\eta(t)\|_{\mathbf{L^2}(\Omega)}\leq1 \quad\text{for a.e. }t\in(0,T).\]
Set $z^k:=\mathds1_\omega\mathscr C_\lambda y^k$ and $z:=\mathds1_\omega\mathscr C_\lambda y$. Since $y^k\to y$ in $C([0,T];H)$, we have $z^k\to z$ in $L^2(0,T;\mathbf{L^2}(\Omega))$.
The weak-$*$ convergence above also implies
$\eta^k\rightharpoonup\eta$ in
$L^2(0,T;\mathbf{L^2}(\Omega))$.
Applying the strong--weak closedness property to the pointwise realization of $\sign_{\mathbf{L^2}(\Omega)}$ in this space, we obtain
\[\eta(t)\in\sign_{\mathbf{L^2}(\Omega)}(\mathds1_\omega\mathscr C_\lambda y(t))\quad\text{for a.e. }t\in(0,T).\]
Using also $y^k\rightharpoonup y$ in $L^2(0,T;V)$ and $d^k\to d$ in $L^1(0,T;\mathbf{L^2}(\omega))$, we pass to the limit in the strong equation in the sense of distributions. Since $L_\lambda y\in L^2(0,T;V')$ and $P(\mathds1_\omega d),DP(\mathds1_\omega\eta)\in L^2(0,T;H)$, we obtain $y'\in L^2(0,T;V')$ and
\[y'+L_\lambda y+DP(\mathds1_\omega\eta)=P(\mathds1_\omega d)\quad\text{in }L^2(0,T;V').\]
Finally, by \Cref{cor:stab-Stokes-strong}, each $y^k$ satisfies
\begin{equation}\label{eq:decay-strong}
\|y^k(t)\|_H \leq C_\lambda e^{-\frac{\lambda}{4}t}\|y_0^k\|_H, \qquad t\in[0,T].
\end{equation}
Passing to the limit using $y^k\to y$ in $C([0,T];H)$
and $y_0^k\to y_0$ in $H$, we obtain
\[\|y(t)\|_H \leq C_\lambda e^{-\frac{\lambda}{4}t}\|y_0\|_H, \qquad t\in[0,T].\]
Since $T>0$ is arbitrary and uniqueness ensures that the
solutions agree on overlapping time intervals, this gives
the global generalized solution and the estimate
\eqref{eq:finalStokes}.
\qed
%---------------------------------------------------

\begin{rem}\label{rem:data-approximation}
In the above proof, we may take $y_0^k=\sum_{j=1}^k(y_0,e_j)_H e_j$. Then $y_0^k\in D(S)$, $y_0^k\to y_0$ in $H$, and
$\|y_0^k\|_H\leq\|y_0\|_H$. The sequence $d^k$ can be obtained by extending $d|_{(0,T)}$ by zero to $\mathbb R$ and convolving in time with a standard family of nonnegative mollifiers. These approximations belong to $W^{1,1}(0,T;\mathbf{L^2}(\omega))$, converge to $d$ in $L^1(0,T;\mathbf{L^2}(\omega))$, and satisfy $\|d^k(t)\|_{\mathbf{L^2}(\omega)}\leq D$.
The same construction will be used in the proofs of Theorems~\ref{thm:sta-NS} and~\ref{thm:sta-NS-3D}. If $y_0\in V$, the spectral truncations also converge in $V$ and satisfy $\|y_0^k\|_V\leq\|y_0\|_V$. The time convolutions converge in $L^2(0,T;\mathbf{L^2}(\omega))$ as well, since the disturbances are uniformly bounded.

\end{rem}
\section{A Priori Stabilization Estimates for the 2D Navier-Stokes System}
\label{sec:sta-NS-2D}
In this section,  we deal with Navier-Stokes system in dimension $n=2$. The estimate derived here relies on the two-dimensional trilinear inequality \eqref{3.56}, which controls the nonlinear contribution by a multiple of $\|y\|_H\|y\|_V^2$. The three-dimensional case is treated separately in \Cref{sec:sta-NS-3D}, where the finite-dimensional projection estimate in \Cref{lem:finite-projection-convection} is combined with a bound in $V$ to establish global strong well-posedness and exponential stabilization. We keep the Lyapunov functional introduced in \eqref{eq:functional}. The linear part of the estimate has already been obtained in the Stokes case. Hence, along strong solutions in the sense of \Cref{def:strong-solution}, the only additional contribution comes from the nonlinear convection term.  More precisely,
\[\frac12\frac{d}{dt}\|y(t)\|_\mu^2+(L_\lambda y(t),y(t))_\mu+D(P(\mathds 1_\omega\eta(t)),y(t))_\mu=
(P(\mathds 1_\omega d(t)),y(t))_\mu-(B(y(t),y(t)),y(t))_\mu,\]
where $\eta \in\sign_{\mathbf{L^2}(\Omega)} (\mathds 1_\omega\mathscr C_\lambda y)$. By the Stokes estimate \eqref{eq:stabilization-Stokes} and \eqref{eq:ineq_sign}, the first three terms yield the same dissipative contribution as before. Hence, the nonlinear Lyapunov estimate becomes
\[\frac{d}{dt}V(y(t))\leq -\frac12\|y(t)\|_V^2-\frac{\lambda}{2}V(y(t))+2\bigl|(B(y(t),y(t)),y(t))_\mu\bigr|.\]
Therefore, it remains only to estimate the nonlinear convection term
\begin{equation}\label{eq:nonlinear_est}
\begin{aligned}
(B(y,y),y)_\mu &=\langle B(y,y),\mathscr C_\lambda y\rangle_{V',V} \\
&=b(y,y,\mathscr C_\lambda y)\\
&= b(y,y,  y)+(\mu_\lambda-1)b(y,y,\mathbb{P}_Ny)\\
&=(\mu_\lambda-1)b(y,y,\mathbb P_Ny).
\end{aligned}    
\end{equation}
By \eqref{eq:cancel-b}, one has that $b(y,y,y)=0.$ The trilinear estimate \eqref{3.56} gives $|b(y,y,\mathbb{P}_Ny)|\le \sqrt{2}\|y\|_H\|y\|_V^2$. Consequently,
\[\bigl|(B(y,y),y)_\mu\bigr|=(\mu_\lambda-1)\bigl|b(y,y,\mathbb P_Ny)\bigr|\leq \sqrt{2}(\mu_\lambda-1)\|y\|_H\|y\|_V^2.\]
Similar to \Cref{sec:sta-Stokes}, we have proved the following a priori stability estimate
\begin{prop}\label{prop:stabNS}
Let $n=2$, $\lambda>0$, and assume that $d$ satisfies \textup{\ref{A_1}--\ref{A_2}}. For every $T>0$, every strong solution $y$ of \eqref{eq:closed_loop_Navier-Stokes}, with initial datum $y(0)=y_0\in V$ and the feedback law $G_\lambda$ defined in \eqref{eq:feedback-Stokes}, satisfies the following a priori estimate for almost every $t\in(0,T)$:
\begin{equation}\label{eq:priori_stab_NS}
\frac{d}{dt}V(y(t))\leq\left(-\frac12+2\sqrt{2}(\mu_\lambda-1)\|y(t)\|_H\right)\|y(t)\|_V^2-\frac{\lambda}{2}V(y(t)).
\end{equation}
\end{prop}
%%%%%%%%%%%%%%%%%%%%%%%%

\section{Well-posedness of the closed-loop Navier--Stokes system}
\label{sec:WP-NS-2D}
Here we consider the closed-loop Navier--Stokes system in dimensions $n=2,3$. In the Stokes case, the operator $A$ is maximal monotone in $H_\mu$, so well-posedness follows from nonlinear maximal monotone theory. The additional convective term prevents a direct application of this argument. To retain the same operator framework, we follow the approximation method of \cite[Section 5.7]{barbu2010}, adapted to the weighted inner product $(\cdot,\cdot)_\mu$. We introduce a truncation of the convective term and prove that the resulting closed-loop operator is maximal monotone after adding a suitable multiple of the identity. We construct strong solutions to the truncated problems in both dimensions. In dimension two, we remove the truncation and combine the resulting well-posedness with the a priori stabilization estimate from \Cref{sec:sta-NS-2D} to prove \Cref{thm:sta-NS}. The three-dimensional well-posedness and stabilization results are established in \Cref{sec:sta-NS-3D}.\smallskip

\noindent We consider the following closed-loop system:
\begin{equation}
\label{eq:closed-loop-NS}
\begin{cases}
\partial_t y-\Delta y+(y\cdot\nabla)y+\nabla p + \gamma_\lambda\mathds 1_\omega \mathbb P_N y+D\mathds 1_\omega\eta =
\mathds 1_\omega d, & \text{in } \Omega\times(0,T),\\
\nabla\cdot y =0, & \text{in }\Omega\times(0,T),\\
\eta\in\sign_{\mathbf{L^2}(\Omega)} (\mathds 1_\omega \mathscr{C}_\lambda y), & \text{a.e. in }(0,T),\\
y=0, & \text{on }\partial\Omega\times(0,T),\\
y(0)=y_0,& \text{in }\Omega.
\end{cases}
\end{equation}
After applying the Leray projector to $\eqref{eq:closed-loop-NS}$, this is equivalent to
\begin{equation}\label{eq:NS-abstract}
y'+Sy+B(y,y)+\gamma_\lambda P(\mathds 1_\omega\mathbb P_Ny)+DP(\mathds 1_\omega\eta)=P(\mathds 1_\omega d),
\end{equation}
with $\eta(t)\in\sign_{\mathbf{L^2}(\Omega)}(\mathds 1_\omega\mathscr C_\lambda y(t))$ for a.e. $t\in(0,T)$.

\subsection*{The truncated problem in dimensions two and three}
For $\delta>0$, let $B_\delta$ be defined by \eqref{eq:definition_B_delta}. Replacing $B$ by $B_\delta$ in \eqref{eq:NS-abstract}, we consider the truncated closed-loop system
\begin{equation}\label{eq:NS-sigma_N}
(y_\delta)'+L_\lambda y_\delta+B_\delta(y_\delta)
+\mathcal B(y_\delta)\ni P(\mathds1_\omega d).
\end{equation}
The truncation $B_\delta$ is introduced to control the nonlinear Navier--Stokes term and to recover the quasi-monotonicity properties needed in the maximal monotone framework. The next lemma establishes quasi-monotonicity in $H_\mu$.
\begin{lemma}
Let $n=2$ or $n=3$, and let $\delta>0$. For every $\varepsilon>0$,
there exists $C_{\delta,\lambda,\varepsilon}>0$ such that
\begin{equation}\label{eq:BN-cuasimonotonicity}
\left|
\left\langle B_\delta(y_1)-B_\delta(y_2),
\mathscr C_\lambda(y_1-y_2)\right\rangle_{V',V}
\right|
\leq \varepsilon\|y_1-y_2\|_V^2
+C_{\delta,\lambda,\varepsilon}\|y_1-y_2\|_H^2,
\qquad y_1,y_2\in V.
\end{equation}
\end{lemma}

\begin{proof}
Notice that, by definition of $\mathscr{C}_\lambda$, one gets
\begin{equation}\label{eq:BN-mu-decomposition}
\begin{aligned}
\langle B_\delta(y_1)-B_\delta(y_2),
\mathscr C_\lambda(y_1-y_2)\rangle_{V',V}=&\langle B_\delta(y_1)-B_\delta(y_2),y_1-y_2\rangle_{V',V}\\
&+(\mu_\lambda-1)\langle B_\delta(y_1)-B_\delta(y_2),\mathbb P_N(y_1-y_2)\rangle_{V',V}.
\end{aligned}
\end{equation}
The estimate proved in \cite[Lemma~5.2, equation~(5.172)]{barbu2010} gives
\[|\langle B_\delta(y_1)-B_\delta(y_2),y_1-y_2\rangle_{V',V}|\leq \frac{\varepsilon}{2}\|y_1-y_2\|_V^2+C_{\delta,\varepsilon}\|y_1-y_2\|_H^2.\]
To estimate the additional weighted term, recall that $R_\delta$ is the truncation operator defined in \Cref{sec:pre} and $B_\delta(y)=B(R_\delta y)$, then $\|R_\delta y\|_V\leq\delta$, and $\|R_\delta y_1-R_\delta y_2\|_V\leq\|y_1-y_2\|_V$. Using \cite[eq.~(5.169), p.~253]{barbu2010}, we obtain
\[\begin{aligned}
\|B_\delta(y_1)-B_\delta(y_2)\|_{V'} &\leq C\bigl(\|R_\delta y_1\|_V+\|R_\delta y_2\|_V\bigr)\|R_\delta y_1-R_\delta y_2\|_V\\
&\leq C\delta\|y_1-y_2\|_V.
\end{aligned}\]
Since $\|\mathbb P_N(y_1-y_2)\|_V\leq\sqrt{\lambda}\|y_1-y_2\|_H$, Young's inequality yields
\[\begin{aligned}
(\mu_\lambda-1)|\langle B_\delta(y_1)-B_\delta(y_2),\mathbb P_N(y_1-y_2)\rangle_{V',V}|
&\leq C\delta(\mu_\lambda-1)\sqrt{\lambda}
\|y_1-y_2\|_V\|y_1-y_2\|_H\\
&\leq \frac{\varepsilon}{2}\|y_1-y_2\|_V^2
+C_{\delta,\lambda,\varepsilon}\|y_1-y_2\|_H^2.
\end{aligned}\]
\end{proof}
For $\delta>0$, we define
\begin{equation*}
\begin{array}{rcl}
\Gamma_\delta:D(S)&\longrightarrow&H,\\
y&\longmapsto&\Gamma_\delta y:=L_\lambda y+B_\delta(y).
\end{array}
\end{equation*}
Recall that $D(L_\lambda)=D(S)$.

\begin{lemma}\label{lem:Gamma-delta-maximal}
For every $\delta>0$, there exists $\beta_\delta>0$
such that both $\Gamma_\delta+\beta_\delta I$ and
$\Gamma_\delta+\beta_\delta I+\mathcal B$
are maximal monotone in $H_\mu$, with domain $D(S)$.
\end{lemma}
\begin{proof}
Let $y_1,y_2\in D(S)$ and set $w:=y_1-y_2$.
Taking $\varepsilon=1/4$ in
\eqref{eq:BN-cuasimonotonicity} and using
\eqref{eq:Llambda-coercive}, we obtain
\[\begin{aligned}
\bigl((\Gamma_\delta+\beta I)y_1-(\Gamma_\delta+\beta I)y_2,w\bigr)_\mu&\geq \frac14\|w\|_V^2-C_{\delta,\lambda}\|w\|_H^2+\beta\|w\|_\mu^2\\
&\geq \frac14\|w\|_V^2,
\end{aligned}\]
whenever $\beta\geq C_{\delta,\lambda}$, since $\|w\|_H\leq\|w\|_\mu$. To prove maximality, we first work in $H$ with its usual
inner product. By \cite[Lemma~5.2]{barbu2010}, there exists
$\alpha_\delta>0$ such that
$S+B_\delta+\alpha_\delta I$, with domain $D(S)$,
is maximal monotone in $H$.
Set
\[K_\lambda:=L_\lambda-S=\gamma_\lambda P(\mathds1_\omega\mathbb P_N)\in\mathcal L(H),\]
and choose $\beta_\delta\geq\max\left\{C_{\delta,\lambda},\alpha_\delta+\|K_\lambda\|_{\mathcal L(H)}\right\}$. Then $K_\lambda+(\beta_\delta-\alpha_\delta)I$ is a bounded monotone linear operator defined on all of $H$, and is therefore maximal monotone in $H$. By the sum theorem
\cite[Corollary~2.7]{Brezis1973},
\[\Gamma_\delta+\beta_\delta I=(S+B_\delta+\alpha_\delta I)+\bigl(K_\lambda+(\beta_\delta-\alpha_\delta)I\bigr)\]
is maximal monotone in $H$. In particular, $R(I+\Gamma_\delta+\beta_\delta I)=H$. Since $H$ and $H_\mu$ have the same underlying vector space, this range condition also holds in $H_\mu$. Together with the monotonicity proved above, Minty's criterion shows that $\Gamma_\delta+\beta_\delta I$ is maximal monotone in $H_\mu$. Finally, $\mathcal B=\partial_\mu\Phi$ is maximal monotone in $H_\mu$ and $D(\mathcal B)=H_\mu$. A further application of
\cite[Corollary~2.7]{Brezis1973} shows that $\Gamma_\delta+\beta_\delta I+\mathcal B$ is maximal monotone
in $H_\mu$, with domain $D(S)$.
\end{proof}
We can now establish the existence and regularity of strong solutions to the truncated problem.
\begin{prop}\label{prop:wp-NS-truncated}
Let $n=2$ or $n=3$, $\delta>0$, $T>0$, and $y_0\in D(S)$. Assume that $d$ satisfies \textup{\ref{A_1}--\ref{A_2}} and $d\in W^{1,1}(0,T;\mathbf{L^2}(\omega))$. Then the truncated problem \eqref{eq:NS-sigma_N}, with initial condition $y_\delta(0)=y_0$,
admits a unique strong solution $y_\delta$ satisfying
\[y_\delta\in W^{1,\infty}(0,T;H) \cap L^\infty(0,T;D(S)) \cap C([0,T];V).\]
Moreover, any two strong solutions of the truncated problem with the same $\delta$ and feedback law, corresponding to data $(y_0,d)$ and $(\widehat y_0,\widehat d)$, satisfy
\begin{equation}\label{eq:truncated-dependence}
\begin{aligned}
\|y_\delta(t)-\widehat y_\delta(t)\|_\mu
\leq e^{\beta_\delta t}\biggl(&\|y_0-\widehat y_0\|_\mu\\
&+\sqrt{\mu_\lambda}\int_0^t\|d(s)-\widehat d(s)\|_{\mathbf{L^2}(\omega)}\,ds\biggr),
\qquad t\in[0,T].
\end{aligned}
\end{equation}
\end{prop}
\begin{proof}
Set $f:=P(\mathds1_\omega d)\in W^{1,1}(0,T;H_\mu)$. By \Cref{lem:Gamma-delta-maximal}, $\Gamma_\delta+\mathcal B+\beta_\delta I$ is maximal monotone in $H_\mu$, with domain $D(S)$. Since $y_0\in D(S)$, \cite[Theorem~4.5, p.~141]{barbu2010}, applied to the quasi-maximal monotone operator $\Gamma_\delta+\mathcal B$ in $H_\mu$ with shift $\beta_\delta$, yields a unique abstract strong solution $y_\delta$ of \eqref{eq:NS-sigma_N}, with $y_\delta(0)=y_0$, satisfying $y_\delta\in W^{1,\infty}(0,T;H_\mu)$, $y_\delta(t)\in D(S)$ for a.e. $t\in(0,T)$. By the equivalence of the norms of $H_\mu$ and $H$, we also have $y_\delta\in W^{1,\infty}(0,T;H)$.
By the definition of $\mathcal B$, every
$\xi\in\mathcal Bv$, $v\in H$, satisfies $\|\xi\|_H\leq D$.
Therefore, \eqref{eq:NS-sigma_N} gives
\[\|\Gamma_\delta y_\delta(t)\|_H\leq \|f(t)\|_H+\|y_\delta'(t)\|_H+D,\qquad\text{for a.e. }t\in(0,T).\]
Since $f\in L^\infty(0,T;H)$ by \ref{A_2} and $y_\delta'\in L^\infty(0,T;H)$, we obtain
\begin{equation}\label{eq:Gamma-delta-Linfty}
\Gamma_\delta y_\delta\in L^\infty(0,T;H).
\end{equation}
By \eqref{eq:B-2d} and Poincar\'e's inequality when $n=2$, and by \eqref{eq:B-3d} when $n=3$, we have $\|B(y)\|_H\leq C\|y\|_V^{3/2}\|Sy\|_H^{1/2}$. The definition of $B_\delta$ therefore gives
\[\|B_\delta(y)\|_H\leq C\delta^{3/2}\|Sy\|_H^{1/2},\qquad y\in D(S).\]
Indeed, for $\|y\|_V>\delta$, the factor $\delta^2\|y\|_V^{-1/2}$ is bounded by $\delta^{3/2}$. Using the definition of $\Gamma_\delta$ and Young's
inequality, for every $y\in D(S)$ we have
\[\begin{aligned}
\|Sy\|_H&\leq \|\Gamma_\delta y\|_H+\gamma_\lambda\|y\|_H+C\delta^{3/2}\|Sy\|_H^{1/2}\\
&\leq \|\Gamma_\delta y\|_H+\gamma_\lambda\|y\|_H+\frac12\|Sy\|_H+C\delta^3.
\end{aligned}\]
Absorbing the term $\frac12\|Sy\|_H$, we obtain
\[\|Sy\|_H\leq 2\|\Gamma_\delta y\|_H+2\gamma_\lambda\|y\|_H+C\delta^3.\]
Since $\Gamma_\delta y_\delta\in L^\infty(0,T;H)$ by \eqref{eq:Gamma-delta-Linfty} and $y_\delta\in L^\infty(0,T;H)$, this gives $Sy_\delta\in L^\infty(0,T;H)$. Therefore, equipping $D(S)$ with its graph norm, we conclude that $y_\delta\in L^\infty(0,T;D(S))$. Finally, the standard embedding $H^1(0,T;H)\cap L^2(0,T;D(S))\hookrightarrow C([0,T];V)$ and the regularity established above imply that $y_\delta\in C([0,T];V)$. Together with the measurable representation of the sign term established above, this regularity shows that $y_\delta$ is a strong solution in the sense of \Cref{def:strong-solution}. For continuous dependence, let $z=y_\delta-\widehat y_\delta$. Subtracting the two equations and using the monotonicity of $\Gamma_\delta+\mathcal B+\beta_\delta I$ in $H_\mu$ yields
\[\frac12\frac{d}{dt}\|z\|_\mu^2\leq\beta_\delta\|z\|_\mu^2+\sqrt{\mu_\lambda}\|d-\widehat d\|_{\mathbf{L^2}(\omega)}\|z\|_\mu.\]
Applying Gronwall's inequality to $(\|z\|_\mu^2+\varepsilon^2)^{1/2}$ and letting $\varepsilon\to0$ proves \eqref{eq:truncated-dependence}. This argument only requires $H^1(0,T;H)$ regularity, so it also applies to any two strong solutions in the sense of \Cref{def:strong-solution}.
\end{proof}

\subsection*{The two-dimensional system}
\begin{prop}\label{thm:wp-NS-strong}
Let $n=2$, $T>0$, and $y_0\in D(S)$. Assume that $d$ satisfies \textup{\ref{A_1}--\ref{A_2}}
and $d\in W^{1,1}(0,T;\mathbf{L^2}(\omega))$. Then the closed-loop Navier--Stokes system \eqref{eq:closed-loop-NS} admits a unique strong solution $y\in W^{1,\infty}(0,T;H) \cap L^\infty(0,T;D(S)) \cap C([0,T];V)$. Moreover, let $y,\widehat y$ be two such solutions corresponding to data $(y_0,d)$ and $(\widehat y_0,\widehat d)$, with the same feedback law
$G_\lambda$ and disturbance bound $D$. For every $R>0$, if $\|y_0\|_H,\|\widehat y_0\|_H\leq R$, then
\begin{equation}\label{eq:NS-continuous-dependence}
\|y-\widehat y\|_{C([0,T];H)}\leq C_{T,R}\left(\|y_0-\widehat y_0\|_H+\|d-\widehat d\|_{L^1(0,T;\mathbf{L^2}(\omega))}\right),
\end{equation}
where $C_{T,R}>0$ depends only on
$T,R,\Omega,\omega,\lambda,D$.
\end{prop}
\begin{proof}
Let $y_\delta$ be the strong solution of \eqref{eq:NS-sigma_N}, with $y_\delta(0)=y_0$, provided by \Cref{prop:wp-NS-truncated}. The idea of the proof is to show that, for $\delta$ sufficiently large, $B_\delta(y_\delta)=B(y_\delta)$ on $[0,T]$, so that $y_\delta$ solves \eqref{eq:closed-loop-NS}. We denote by $\xi_\delta(t)\in\mathcal B(y_\delta(t))$ the selection appearing in the equation. Recall that $\|\xi_\delta(t)\|_H\leq D$ for almost every $t\in(0,T)$. Testing \eqref{eq:NS-sigma_N} with $y_\delta$ in $H$ and using that $b(y_\delta,y_\delta,y_\delta)=0$, we obtain, for almost every $t\in(0,T)$,
\[\frac12\frac{d}{dt}\|y_\delta(t)\|_H^2+\|y_\delta(t)\|_V^2+\gamma_\lambda
(\mathds1_\omega\mathbb P_Ny_\delta(t),y_\delta(t))_{\mathbf{L^2}(\Omega)}+(\xi_\delta(t),y_\delta(t))_H=(\mathds1_\omega d(t),y_\delta(t))_{\mathbf{L^2}(\Omega)}.\]
We obtain the following bounds:
\[\left|\gamma_\lambda(\mathds1_\omega\mathbb P_Ny_\delta(t),y_\delta(t))
_{\mathbf{L^2}(\Omega)}\right|\leq\gamma_\lambda\|y_\delta(t)\|_H^2,\qquad|(\xi_\delta(t),y_\delta(t))_H|
\leq\frac12\|y_\delta(t)\|_H^2+\frac{D^2}{2},\]
\[|(\mathds1_\omega d(t),y_\delta(t))_{\mathbf{L^2}(\Omega)}|\leq\frac12\|d(t)\|_{\mathbf{L^2}(\omega)}^2+\frac12\|y_\delta(t)\|_H^2.\]
Therefore,
\begin{equation}
\label{eq:energy-delta}
\frac12\frac{d}{dt}\|y_\delta(t)\|_H^2+\|y_\delta(t)\|_V^2
\leq(\gamma_\lambda+1)\|y_\delta(t)\|_H^2+\frac{D^2}{2}+\frac12\|d(t)\|_{\mathbf{L^2}(\omega)}^2,\quad\text{for a.e. }t\in(0,T).
\end{equation}
Here $\gamma_\lambda>0$ is independent of $\delta$. By Gronwall's lemma and \ref{A_2},
\begin{equation}
\label{eq:bound_gronwall_CT_1}
\sup_{t\in[0,T]}\|y_\delta(t)\|_H^2+\int_0^T\|y_\delta(t)\|_V^2\,dt\leq C_T,
\end{equation}
where $C_T>0$ depends on $T,\lambda,D,\|y_0\|_H$, but is independent of $\delta$.
Thus $\{y_\delta\}$ is bounded in $L^\infty(0,T;H)\cap L^2(0,T;V)$. Now, testing \eqref{eq:NS-sigma_N} with $Sy_\delta$ in $H$, we obtain, for almost every $t\in(0,T)$,
\[\begin{aligned}
\frac12\frac{d}{dt}\|y_\delta(t)\|_V^2+\|Sy_\delta(t)\|_H^2
+\gamma_\lambda(\mathds1_\omega\mathbb P_Ny_\delta(t),Sy_\delta(t))
_{\mathbf{L^2}(\Omega)}+(B_\delta(y_\delta(t)),Sy_\delta(t))_H\\
+(\xi_\delta(t),Sy_\delta(t))_H=(\mathds1_\omega d(t),Sy_\delta(t))_{\mathbf{L^2}(\Omega)}.
\end{aligned}\]
Following the same procedure as before, we obtain
\[\left|\gamma_\lambda(\mathds1_\omega\mathbb P_Ny_\delta(t),Sy_\delta(t))_{\mathbf{L^2}(\Omega)}\right|\leq2\gamma_\lambda^2\|y_\delta(t)\|_H^2+\frac18\|Sy_\delta(t)\|_H^2,\]
\[|(\xi_\delta(t),Sy_\delta(t))_H|\leq\frac18\|Sy_\delta(t)\|_H^2+2D^2,\]
\[|(\mathds1_\omega d(t),Sy_\delta(t))_{\mathbf{L^2}(\Omega)}|\leq2\|d(t)\|_{\mathbf{L^2}(\omega)}^2+\frac18\|Sy_\delta(t)\|_H^2.\]
For the nonlinear term, the two-dimensional estimate for $B$, the definition of $B_\delta$, and Young's inequality give
\[\begin{aligned}
|(B_\delta(y_\delta(t)),Sy_\delta(t))_H|&\leq C\|y_\delta(t)\|_H^{1/2}\|y_\delta(t)\|_V\|Sy_\delta(t)\|_H^{3/2}\\
&\leq\frac18\|Sy_\delta(t)\|_H^2+C\|y_\delta(t)\|_H^2\|y_\delta(t)\|_V^4\\
&\leq\frac18\|Sy_\delta(t)\|_H^2+CC_T\|y_\delta(t)\|_V^4,
\end{aligned}\]
where we used \eqref{eq:bound_gronwall_CT_1} and $C>0$ depends only on $\Omega$. Integrating in $(0,t)$, it follows that
\[\begin{aligned}
\|y_\delta(t)\|_V^2+\int_0^t\|Sy_\delta(s)\|_H^2\,ds\leq&\|y_0\|_V^2+4\int_0^T\left(\gamma_\lambda^2\|y_\delta(s)\|_H^2+D^2+\|d(s)\|_{\mathbf{L^2}(\omega)}^2\right)\,ds\\
&+K_T\int_0^t\|y_\delta(s)\|_V^4\,ds,
\end{aligned}\]
where $K_T:=2CC_T$ is independent of $\delta$. Using \eqref{eq:bound_gronwall_CT_1} and \ref{A_2}, we obtain
\[\|y_\delta(t)\|_V^2+\int_0^t\|Sy_\delta(s)\|_H^2\,ds\leq\widetilde C_T+K_T\int_0^t\|y_\delta(s)\|_V^4\,ds,\]
where $\widetilde C_T:=\|y_0\|_V^2+4T(\gamma_\lambda^2C_T+2D^2)$ is independent of $\delta$. Since $\|y_\delta(s)\|_V^4=\|y_\delta(s)\|_V^2\|y_\delta(s)\|_V^2$, we may write
\[\|y_\delta(t)\|_V^2\leq\widetilde C_T+K_T\int_0^t\|y_\delta(s)\|_V^2\|y_\delta(s)\|_V^2\,ds.\]
On the other hand, \eqref{eq:bound_gronwall_CT_1} yields $\int_0^T\|y_\delta(s)\|_V^2\,ds\leq C_T$, uniformly with respect to $\delta$. Therefore, by the Gronwall--Bellman inequality
\cite[Lemma~1, p.~644]{Bellman1943},
\[\|y_\delta(t)\|_V^2\leq\widetilde C_T\exp\left(K_T\int_0^t\|y_\delta(s)\|_V^2\,ds\right)\leq\widetilde C_Te^{K_TC_T},\qquad t\in[0,T].\]
Moreover,
\[\int_0^T\|y_\delta(s)\|_V^4\,ds\leq\sup_{t\in[0,T]}\|y_\delta(t)\|_V^2\int_0^T\|y_\delta(s)\|_V^2\,ds\leq C_T\widetilde C_Te^{K_TC_T}.\]
Coming back to the previous energy estimate, we obtain
\[\sup_{t\in[0,T]}\|y_\delta(t)\|_V^2+\int_0^T\|Sy_\delta(s)\|_H^2\,ds\leq\widetilde C_T+(1+K_TC_T)\widetilde C_Te^{K_TC_T}.\]
Therefore, choosing $\delta>\sqrt{\widetilde C_T}\,e^{K_TC_T/2}$, we have $\|y_\delta(t)\|_V<\delta$ for every $t\in[0,T]$. By \eqref{eq:definition_B_delta}, $B_\delta(y_\delta)=B(y_\delta)$ on $[0,T]$. Hence $y:=y_\delta$ is a strong solution of
\eqref{eq:closed-loop-NS}, with the stated regularity by \Cref{prop:wp-NS-truncated}. We conclude by proving continuous dependence on the data and uniqueness. Let $y_1,y_2$ be two strong solutions corresponding to data $(y_{0,1},d_1)$ and $(y_{0,2},d_2)$ satisfying the assumptions of the proposition, with the same feedback law $G_\lambda$
and disturbance bound $D$. Let $\eta_i(t)\in\sign_{\mathbf{L^2}(\Omega)}(\mathds1_\omega\mathscr C_\lambda y_i(t))$, $i=1,2$,
be the selections appearing in their respective equations. Set $z:=y_1-y_2$. Then
\[z'+Sz+B(y_1)-B(y_2)+\gamma_\lambda P(\mathds1_\omega\mathbb P_Nz)+DP(\mathds1_\omega(\eta_1-\eta_2))=P(\mathds1_\omega(d_1-d_2)).\]
Taking the $\mu$-inner product with $z$, we obtain, for almost every $t\in(0,T)$,
\[\begin{aligned}
\frac12\frac{d}{dt}\|z(t)\|_\mu^2&+(L_\lambda z(t),z(t))_\mu+D(P(\mathds1_\omega(\eta_1(t)-\eta_2(t))),z(t))_\mu\\
&=-\langle B(y_1)-B(y_2),\mathscr C_\lambda z\rangle_{V',V}+(P(\mathds1_\omega(d_1-d_2)),z)_\mu.
\end{aligned}\]
By \eqref{eq:Llambda-coercive}, $(L_\lambda z,z)_\mu\geq\frac12\|z\|_V^2$. Moreover, by the monotonicity of the sign operator,
\[\begin{aligned}
D(P(\mathds1_\omega(\eta_1-\eta_2)),z)_\mu&=D(\eta_1-\eta_2,\mathds1_\omega\mathscr C_\lambda z)_{\mathbf{L^2}(\Omega)}\\
&=D(\eta_1-\eta_2,\mathds1_\omega\mathscr C_\lambda y_1-\mathds1_\omega\mathscr C_\lambda y_2)_{\mathbf{L^2}(\Omega)}\geq0.
\end{aligned}\]
The disturbance term satisfies
\[|(P(\mathds1_\omega(d_1-d_2)),z)_\mu|\leq\sqrt{\mu_\lambda}\|d_1-d_2\|_{\mathbf{L^2}(\omega)}\|z\|_\mu.\]
Therefore,
\[\frac12\frac{d}{dt}\|z(t)\|_\mu^2+\frac12\|z(t)\|_V^2\leq |\langle B(y_1)-B(y_2),\mathscr C_\lambda z\rangle_{V',V}|+\sqrt{\mu_\lambda}\|d_1-d_2\|_{\mathbf{L^2}(\omega)}\|z\|_\mu.\]
Since $B(y_1)-B(y_2)=B(z,y_1)+B(y_2,z)$ and $b(y_2,z,z)=0$, the trilinear estimates and $\|\mathbb P_Nz\|_V\leq\sqrt{\lambda}\|z\|_H$ give
\[\begin{aligned}
|\langle B(y_1)-B(y_2),\mathscr C_\lambda z\rangle_{V',V}|
&\leq |b(z,y_1,z)|
+(\mu_\lambda-1)\|B(y_1)-B(y_2)\|_{V'}
\|\mathbb P_Nz\|_V\\
&\leq\widetilde C_\lambda
(\|y_1\|_V+\|y_2\|_V)\|z\|_H\|z\|_V,
\end{aligned}\]
for some constant $\widetilde C_\lambda>0$.
By Young's inequality,
\[|\langle B(y_1)-B(y_2),\mathscr C_\lambda z\rangle_{V',V}|
\leq\frac14\|z\|_V^2
+2\widetilde C_\lambda^2
(\|y_1\|_V^2+\|y_2\|_V^2)\|z\|_H^2.\]
Since $\|z\|_H\leq\|z\|_\mu$, we obtain
\[\frac12\frac{d}{dt}\|z(t)\|_\mu^2+\frac14\|z(t)\|_V^2\leq2\widetilde C_\lambda^2(\|y_1(t)\|_V^2+\|y_2(t)\|_V^2)\|z(t)\|_\mu^2+\sqrt{\mu_\lambda}\|d_1(t)-d_2(t)\|_{\mathbf{L^2}(\omega)}\|z(t)\|_\mu.\]
Applying Gronwall's inequality to $(\|z(t)\|_\mu^2+\varepsilon^2)^{1/2}$ and letting $\varepsilon\to0$, we obtain
\[\begin{aligned}
\|z(t)\|_\mu\leq&\exp\left(2\widetilde C_\lambda^2\int_0^t(\|y_1(s)\|_V^2+\|y_2(s)\|_V^2)\,ds\right)\\
&\times\left(\|y_{0,1}-y_{0,2}\|_\mu+\sqrt{\mu_\lambda}\int_0^t\|d_1(s)-d_2(s)\|_{\mathbf{L^2}(\omega)}\,ds\right).
\end{aligned}\]
The energy estimate \eqref{eq:bound_gronwall_CT_1} also holds for $y_1$ and $y_2$.
If $\|y_{0,1}\|_H,\|y_{0,2}\|_H\leq R$, this estimate bounds the exponential factor uniformly in terms of
$T,R,\Omega,\omega,\lambda,D$. Using $\|z(t)\|_H\leq\|z(t)\|_\mu$ and $\|z(0)\|_\mu\leq\sqrt{\mu_\lambda}\|z(0)\|_H$ proves \eqref{eq:NS-continuous-dependence}. In particular, if $y_{0,1}=y_{0,2}$ and $d_1=d_2$, then $z(t)=0$ for every $t\in[0,T]$.
Thus the strong solution is unique.
\end{proof}

\medskip
\noindent\emph{Proof of \Cref{thm:sta-NS}.}
Fix $\lambda>0$ and consider the feedback law $G_\lambda$ defined in \eqref{eq:feedback-Stokes}. Recall that $C_\lambda=\sqrt{\mu_\lambda}$, and set
\[r_\lambda:=\frac{1}{4\sqrt{2}(\mu_\lambda-1)},\qquad\rho_\lambda:=\frac{r_\lambda}{\sqrt{\mu_\lambda}}.\]
Let $T>0$, $y_0\in H$ with $\|y_0\|_H\leq\rho_\lambda$, and let $d$ satisfy \textup{\ref{A_1}--\ref{A_2}}. Choose $y_0^k\in D(S)$ such that $y_0^k\to y_0$ in $H$ and $\|y_0^k\|_H\leq\|y_0\|_H$. Choose also $d^k\in W^{1,1}(0,T;\mathbf{L^2}(\omega))$ such that
\[d^k\to d\quad\text{in }L^1(0,T;\mathbf{L^2}(\omega)),\qquad\|d^k(t)\|_{\mathbf{L^2}(\omega)}\leq D\quad\text{for a.e. }t\in(0,T).\]
By \Cref{thm:wp-NS-strong}, the corresponding strong solutions $y^k$ exist on $[0,T]$. We first establish their exponential decay. As long as $\|y^k(t)\|_H<r_\lambda$, the a priori estimate \eqref{eq:priori_stab_NS} gives
\[\frac{d}{dt}V(y^k(t))\leq-\frac{\lambda}{2}V(y^k(t)).\]
Consequently,
\[\|y^k(t)\|_H^2\leq V(y^k(t))\leq e^{-\lambda t/2}V(y_0^k)\leq e^{-\lambda t/2}\mu_\lambda\rho_\lambda^2<\mu_\lambda\rho_\lambda^2=r_\lambda^2.\]
Since $y^k\in C([0,T];H)$, this excludes a first time at which $\|y^k(t)\|_H=r_\lambda$. Thus the estimate holds throughout $[0,T]$, and
\begin{equation}
\label{eq:NS-decay-strong}
\|y^k(t)\|_H\leq\sqrt{\mu_\lambda}e^{-\lambda t/4}\|y_0^k\|_H,\qquad t\in[0,T].
\end{equation}
By \eqref{eq:NS-continuous-dependence},
\[\|y^k-y^\ell\|_{C([0,T];H)}\leq C_{T,\rho_\lambda}\left(\|y_0^k-y_0^\ell\|_H+\|d^k-d^\ell\|_{L^1(0,T;\mathbf{L^2}(\omega))}\right).\]
Hence $\{y^k\}$ is a Cauchy sequence in $C([0,T];H)$. Denote its limit by $y$. Then $y(0)=y_0$, and $y$ is a generalized solution obtained by approximation with strong solutions, as in the Stokes case. The same continuous dependence estimate shows that the limit is independent of the approximating sequences and that the generalized solution is unique. Moreover, \eqref{eq:bound_gronwall_CT_1} gives a bound for $\{y^k\}$ in $L^2(0,T;V)$, independent of $k$. Extracting a subsequence and using the convergence in $C([0,T];H)$ to identify its weak limit, we obtain $y^k\rightharpoonup y$ in $L^2(0,T;V)$.
Therefore, $y\in C([0,T];H)\cap L^2(0,T;V)$. The two-dimensional interpolation inequality gives
\[\|y^k-y\|_{L^4(0,T;\mathbf{L^4}(\Omega))}^4\leq C\|y^k-y\|_{C([0,T];H)}^2\|y^k-y\|_{L^2(0,T;V)}^2\longrightarrow0.\]
Since $\{y^k\}$ is bounded in $L^4(0,T;\mathbf{L^4}(\Omega))$, the estimate 
\[\|B(u)-B(v)\|_{V'}\leq C(\|u\|_{\mathbf{L^4}(\Omega)}+\|v\|_{\mathbf{L^4}(\Omega)})\|u-v\|_{\mathbf{L^4}(\Omega)}\]
and H\"older's inequality yield $B(y^k)\to B(y)$ in $L^2(0,T;V')$. Arguing as in the proof of \Cref{thm:sta-Stokes}, we extract a subsequence of the sign selections converging weakly-$*$ in $L^\infty(0,T;\mathbf{L^2}(\Omega))$ to a function $\eta$ satisfying
\[\eta(t)\in\sign_{\mathbf{L^2}(\Omega)}(\mathds1_\omega\mathscr C_\lambda y(t))\quad\text{for a.e. }t\in(0,T).\]
Passing to the limit in \eqref{eq:NS-abstract} in the sense of distributions, and observing that all terms other than $y'$ belong to $L^2(0,T;V')$, we obtain $y'\in L^2(0,T;V')$ and
\[y'+L_\lambda y+B(y)+DP(\mathds1_\omega\eta)=P(\mathds1_\omega d)\quad\text{in }L^2(0,T;V').\]
Passing to the limit in \eqref{eq:NS-decay-strong}, using $y^k\to y$ in $C([0,T];H)$ and $y_0^k\to y_0$ in $H$, we conclude that
\[\|y(t)\|_H\leq C_\lambda e^{-\lambda t/4}\|y_0\|_H,\qquad t\in[0,T].\]
Since $T>0$ is arbitrary and uniqueness ensures that the solutions agree on overlapping time intervals, we obtain a unique global generalized solution $y\in C([0,\infty);H)\cap L^2_{\mathrm{loc}}(0,\infty;V)$ satisfying the claimed exponential decay estimate.
\qed

\section{Rapid stabilization of the three-dimensional Navier--Stokes system}
\label{sec:sta-NS-3D}
In this section, we prove \Cref{thm:sta-NS-3D}, establishing global strong well-posedness and rapid exponential stabilization for initial data satisfying \eqref{eq:smallness-NS-3D}. We consider \eqref{eq:NS-abstract} in dimension $n=3$, with the same feedback law \eqref{eq:feedback-Stokes}. The truncated solutions are already provided by \Cref{prop:wp-NS-truncated}. We now combine the weighted Lyapunov estimate in $H$ with a uniform bound in $V$ to remove the truncation. The three-dimensional energy estimate contains a term of order $\|y\|_V^6$. If the Stokes dissipation is reduced to a multiple of $\|y\|_V^2$, following line by line the small invariant ball argument of \cite[Section~5.7]{barbu2010} also restricts the size of the forcing. In our case that means asking for small perturbations, which is not interesting in our setting. Here, the disturbance is canceled in the weighted $H$ estimate. We retain $\|Sy\|_H^2$ in the $V$ estimate and use
\[\|y\|_V^4\leq\|y\|_H^2\|Sy\|_H^2.\]
Smallness in $H$ then provides enough dissipation to keep the solution in a prescribed bounded set in $V$, for any fixed disturbance bound $D$.
\begin{proof}[Proof of \Cref{thm:sta-NS-3D}]
Fix $T>0$. We first assume that $y_0\in D(S)$ and $d\in W^{1,1}(0,T;\mathbf{L^2}(\omega))$. Fix $\delta>R$ and let $y_\delta$ be the strong solution provided by \Cref{prop:wp-NS-truncated}. We derive estimates independent of $\delta$ and then show that the truncation remains inactive. Writing $B_\delta(y)=\theta_\delta(y)B(y)$, where $0\leq\theta_\delta(y)\leq1$, \Cref{lem:finite-projection-convection} gives
\begin{equation}\label{eq:convection-Lyapunov-3D}
|\langle B_\delta(y),\mathscr C_\lambda y\rangle_{V',V}|
\leq a_\lambda\|y\|_H^3,\qquad y\in V.
\end{equation}
Let
\[E_\delta(t):=\|y_\delta(t)\|_\mu^2,\qquad X_\delta(t):=\|y_\delta(t)\|_V^2,\qquad Z_\delta(t):=\|Sy_\delta(t)\|_H^2.\]
The estimate for the linear part obtained in \Cref{sec:sta-Stokes}, together with the cancellation between the disturbance and the sign feedback, gives
\begin{equation}\label{eq:weighted-energy-3D}
E_\delta'+\frac12X_\delta+\frac\lambda2E_\delta\leq 2a_\lambda\|y_\delta\|_H^3\quad\text{for a.e. }t\in(0,T).
\end{equation}
We next estimate the $V$ norm. Let $\xi_\delta(t)\in\mathcal B(y_\delta(t))$ be the selection appearing in the equation, and set
\[F_\delta:=P(\mathds1_\omega d)-\gamma_\lambda P(\mathds1_\omega\mathbb P_Ny_\delta)-\xi_\delta.\]
Then
\[y_\delta'+Sy_\delta+B_\delta(y_\delta)=F_\delta,\qquad \|F_\delta\|_H\leq 2D+\gamma_\lambda\|y_\delta\|_H.\]
Testing this equation with $Sy_\delta$ and using \eqref{eq:B-3d},
\[|(B_\delta(y_\delta),Sy_\delta)_H|\leq C\|y_\delta\|_V^{3/2}\|Sy_\delta\|_H^{3/2},\]
Young's inequality yields a constant $C_*>0$, depending only on $\Omega$, such that
\begin{equation}\label{eq:V-energy-3D}
X_\delta'+Z_\delta \leq C_*X_\delta^3 +C_*\bigl(D^2+\gamma_\lambda^2\|y_\delta\|_H^2\bigr) \quad\text{for a.e. }t\in(0,T).
\end{equation}
Choose $h>0$ sufficiently small that
\begin{equation}\label{eq:h-choice-3D}
\frac{2a_\lambda h}{\tau_1}\leq\frac14,\qquad C_*R^2h^2\leq\frac12,\qquad C_*h^2(D^2+\gamma_\lambda^2h^2)\leq\frac{R^4}{8},
\end{equation}
and define
\[\rho_{\lambda,D,R}:=\frac{h}{2\sqrt{\mu_\lambda}}.\]
Such a choice is possible for every $\lambda,D,R>0$. We now close the bootstrap argument
\begin{equation}\label{eq:bootstrap-3D}
\|y_\delta(t)\|_H\leq h,\qquad X_\delta(t)\leq R^2.
\end{equation}
The initial conditions \eqref{eq:smallness-NS-3D} give
\[E_\delta(0)\leq\frac{h^2}{4},\qquad X_\delta(0)\leq\frac{R^2}{4}.\]
Define
\[T_0:=\sup\bigl\{t\in[0,T]:\|y_\delta(s)\|_H\leq h,\quad X_\delta(s)\leq R^2
\ \text{for every }s\in[0,t]\bigr\}.\]
The strict initial bounds and the continuity of $y_\delta$ in $V$
imply that $T_0>0$. By continuity, \eqref{eq:bootstrap-3D}
holds on $[0,T_0]$. By the spectral decomposition,
\[\|y_\delta\|_H^3\leq \frac{\|y_\delta\|_H}{\tau_1}X_\delta.\]
Thus, \eqref{eq:weighted-energy-3D} and the first condition in \eqref{eq:h-choice-3D} imply
\[E_\delta'+\frac14X_\delta+\frac\lambda2E_\delta\leq0.\]
Consequently,
\begin{equation}\label{eq:weighted-decay-3D}
e^{\lambda t/2}E_\delta(t)
+\frac14\int_0^t e^{\lambda s/2}X_\delta(s)\,ds
\leq E_\delta(0),
\end{equation}
and in particular
\[\|y_\delta(t)\|_H\leq \sqrt{\mu_\lambda}e^{-\lambda t/4}\|y_0\|_H\leq \frac h2.\]
To improve the bound in $V$, we retain the full Stokes dissipation. For almost every $t$,
\begin{equation}\label{eq:interpolation-3D}
X_\delta
=(Sy_\delta,y_\delta)_H
\leq \sqrt{Z_\delta}\,\|y_\delta\|_H,
\qquad
X_\delta^2\leq h^2Z_\delta.
\end{equation}
Therefore, on the bootstrap interval,
\[C_*X_\delta^3 \leq C_*R^2h^2Z_\delta \leq\frac12Z_\delta.\]
Using \eqref{eq:V-energy-3D}, we obtain
\[X_\delta'+\frac12Z_\delta\leq C_*(D^2+\gamma_\lambda^2h^2).\]
Combining this inequality with \eqref{eq:interpolation-3D} and the last condition in \eqref{eq:h-choice-3D} gives
\begin{equation}\label{eq:scalar-barrier-3D}
X_\delta'+\frac{X_\delta^2}{2h^2}
\leq\frac{R^4}{8h^2}.
\end{equation}
The constant function $R^2/2$ solves the corresponding scalar differential equation. Since $X_\delta(0)\leq R^2/4$, scalar comparison yields
\[X_\delta(t)\leq\frac{R^2}{2}.\]
Thus, for every $t\in[0,T_0]$,
\[\|y_\delta(t)\|_H\leq \frac h2<h,\qquad
X_\delta(t)\leq \frac{R^2}{2}<R^2.\]
If $T_0<T$, continuity would extend both bootstrap bounds
to a larger initial interval, contradicting the definition of $T_0$. Hence $T_0=T$, and these estimates hold
throughout $[0,T]$. Since $\delta>R$, we have $\|y_\delta(t)\|_V<\delta$ for every $t\in[0,T]$. Hence $B_\delta(y_\delta)=B(y_\delta)$, and $y_\delta$ solves the original closed-loop system. This proves existence and the stated estimates for regular data. We now consider $y_0\in V$ satisfying
\eqref{eq:smallness-NS-3D} and a disturbance satisfying only \textup{\ref{A_1}--\ref{A_2}}. Fix $T>0$. As in \Cref{rem:data-approximation}, choose $y_0^k\in D(S)$ and $d^k\in W^{1,1}(0,T;\mathbf{L^2}(\omega))$ such that
\[y_0^k\to y_0\quad\text{in }V,\qquad d^k\to d\quad\text{in }L^2(0,T;\mathbf{L^2}(\omega)),\]
with
\[\|y_0^k\|_H\leq\|y_0\|_H,\qquad\|y_0^k\|_V\leq\|y_0\|_V,\qquad\|d^k(t)\|_{\mathbf{L^2}(\omega)}\leq D.\]
Let $y^k$ be the corresponding strong solutions. All these solutions satisfy the estimates above and solve the same truncated problem for the fixed choice $\delta>R$. By \eqref{eq:truncated-dependence}, $\{y^k\}$ is Cauchy in $C([0,T];H)$. Denote its limit by $y$. The preceding estimates also give
\[\sup_{t\in[0,T]}\|y^k(t)\|_V^2\leq\frac{R^2}{2},\qquad \frac12\int_0^T\|Sy^k(s)\|_H^2 ds\leq\|y_0^k\|_V^2+C_*T(D^2+\gamma_\lambda^2h^2).\]
Moreover, the three-dimensional convection estimate yields
\[\int_0^T\|B(y^k(s))\|_H^2 ds\leq CR^3T^{1/2}\left(\int_0^T\|Sy^k(s)\|_H^2 ds\right)^{1/2}.\]
The equation and the uniform bound on the sign feedback therefore show that $\{(y^k)'\}$ is bounded in $L^2(0,T;H)$. Extracting a subsequence, we obtain
\[y^k\rightharpoonup y \quad\text{in }H^1(0,T;H)\cap L^2(0,T;D(S)).\]
In particular, $y\in C([0,T];V)$. The interpolation identity also gives
\[\|y^k-y\|_{L^2(0,T;V)}^2\leq\|S(y^k-y)\|_{L^2(0,T;H)} \|y^k-y\|_{L^2(0,T;H)} \longrightarrow0.\]
Using the uniform bounds in $L^\infty(0,T;V)$ and the local Lipschitz continuity of $B:V\to V'$, we conclude that
\[B(y^k)\to B(y)\quad\text{in }L^2(0,T;V').\]
Together with the uniform bound in $L^2(0,T;H)$, this identifies the weak limit of $B(y^k)$ in that space with $B(y)$. As in the proof of \Cref{thm:sta-Stokes}, choose measurable sign selections $\eta^k$ appearing in the equations for $y^k$. Up to a subsequence,
\[\eta^k\rightharpoonup^\ast\eta \quad\text{in }L^\infty(0,T;\mathbf{L^2}(\Omega)).\]
Since $\mathds1_\omega\mathscr C_\lambda y^k \to\mathds1_\omega\mathscr C_\lambda y$ strongly in $L^2(0,T;\mathbf{L^2}(\Omega))$, the strong--weak closedness of the pointwise realization of the sign operator gives
\[\eta(t)\in\sign_{\mathbf{L^2}(\Omega)}(\mathds1_\omega\mathscr C_\lambda y(t))\quad\text{for a.e. }t\in(0,T).\]
Passing to the limit in the equations yields
\[y'+Sy+B(y)+\gamma_\lambda P(\mathds1_\omega\mathbb P_Ny)+DP(\mathds1_\omega\eta)=P(\mathds1_\omega d)\quad\text{in }L^2(0,T;H).\]
Thus $y$ is a strong solution with the asserted regularity.
The convergence in $C([0,T];H)$ gives
\eqref{eq:decay-NS-3D}. Weak lower semicontinuity and continuity in $V$ give
\[\sup_{t\in[0,T]}\|y(t)\|_V\leq\frac{R}{\sqrt2}.\]
Finally, let $y$ and $\widehat y$ be two strong solutions with the same data on $[0,T]$. Since both belong to $C([0,T];V)$, we may choose $\delta$ larger than their $V$ norms throughout this interval. Both solutions then solve the same truncated problem. The monotonicity calculation used to prove
\eqref{eq:truncated-dependence} applies to their $H^1(0,T;H)$ regularity and gives $y=\widehat y$. Since $T>0$ is arbitrary, uniqueness allows us to join
the solutions on overlapping intervals and obtain the global strong solution. For regular data, this solution coincides with the one constructed in the first part of the proof, which gives the additional $W^{1,\infty}(0,T;H)\cap L^\infty(0,T;D(S))$ regularity.
\end{proof}

\begin{rem}
We do not need smallness in $V$, the bound $R$ and the disturbance bound $D$ may be arbitrarily large, while the admissible radius $\rho_{\lambda,D,R}$ depends on both. Since $\|S^{1/4}y_0\|_H^2\leq\|y_0\|_H\|y_0\|_V$, smallness in $H$ with a fixed bound in $V$ also yields smallness in the critical space $D(S^{1/4})$; compare \cite{bt2004} and \cite[Section~1.4.1]{lasieckapriyasadtriggiani2021}.
\end{rem}

\begin{rem}
Recently, in \cite{chaves2025optimal} a new spectral inequality with vanishing components has been proved. Thanks to this result, one can obtain  analogous stabilization results in the case where the control and the disturbance act only on $(n-1)$ components. More precisely, the control may be chosen so that $u_j\equiv 0$  for some $j\in{1,\dots,n}$.
\end{rem}
\bibliographystyle{siam}   % o alpha, siam, amsplain, etc.
\bibliography{bibliography}
\end{document}